\documentclass[a4paper,ukenglish,cleveref, autoref]{lipics-v2019}

\usepackage{todonotes}
\usepackage{mathtools}
\usepackage[utf8]{inputenc}
\usepackage{amssymb}
\usepackage{amsfonts}
 \usepackage{amsthm}
\usepackage{fancyhdr}
\usepackage{stmaryrd}
\usepackage{xspace}
\usepackage{tikz}
\usetikzlibrary{positioning}
\usetikzlibrary{decorations.pathreplacing}
\usetikzlibrary{arrows.meta}
\usepackage{multicol}
 \usepackage{array}
 \usepackage{booktabs}
 \usepackage{tabularx,hhline,rotating,xspace}
 \usepackage{xparse}
\usetikzlibrary{calc}
\usepackage{bm}
\usepackage{todonotes}

\makeatletter
\newcommand{\thickhline}{%
    \noalign {\ifnum 0=`}\fi \hrule height 1pt
    \futurelet \reserved@a \@xhline
}
\newcolumntype{"}{@{\hskip\tabcolsep\vrule width 1pt\hskip\tabcolsep}}
\makeatother
\newcommand{\TC}{\mathsf{TC}}
\newcommand{\NC}{\mathsf{NC}}
\newcommand{\coNP}{\mathsf{coNP}}
\newcommand{\NP}{\mathsf{NP}}
\newcommand{\Ptime}{\mathsf{P}}
\newcommand{\Sol}{\mathrm{sol}}

\newcommand{\supp}{\mathrm{supp}}

\newcommand{\SAT}{\textsc{Sat}}
\newcommand{\PowWP}{\textsc{PowerWP}}
\newcommand{\PowPP}{\textsc{PowerPP}}
\newcommand{\Int}{\mathrm{Int}}
\newcommand{\KP}{\textsc{Knapsack}}

\newcommand{\ExpEq}{\textsc{ExpEq}}

\DeclareDocumentCommand{\Powerset}{O{} m}{\mathbb{P}_{#1}(#2)}
\newcommand{\norm}[1]{\|#1\|}
\newcommand{\N}{\mathbb{N}}
\newcommand{\Z}{\mathbb{Z}}
\renewcommand{\O}{\mathcal{O}}

\newcommand{\WP}{\textsc{WP}\xspace}
\newcommand{\ms}[1]{\{\!\!\{#1\}\!\!\}}

\title{The complexity of knapsack problems in wreath products}

\author{Michael Figelius}{Universit{\"a}t Siegen, Germany }{}{}{Funded by DFG project LO 748/12-1.}

\author{Moses Ganardi}{Universit{\"a}t Siegen, Germany }{}{}{}

\author{Markus Lohrey}{Universit{\"a}t Siegen, Germany }{}{}{Funded by DFG project LO 748/12-1.}

\author{Georg Zetzsche}{MPI f\"ur Software Systeme, Kaiserslautern, Germany}{}{}{} 

\authorrunning{M. Figelius, M. Ganardi, M. Lohrey and G. Zetzsche}

\Copyright{Michael Figelius, Moses Ganardi, Markus Lohrey and Georg Zetzsche}



\category{}

\relatedversion{}

\supplement{}

\acknowledgements{}

\nolinenumbers 

\EventEditors{John Q. Open and Joan R. Access}
\EventNoEds{2}
\EventLongTitle{42nd Conference on Very Important Topics (CVIT 2016)}
\EventShortTitle{CVIT 2016}
\EventAcronym{CVIT}
\EventYear{2016}
\EventDate{December 24--27, 2016}
\EventLocation{Little Whinging, United Kingdom}
\EventLogo{}
\SeriesVolume{42}
\ArticleNo{23}

\begin{document}
	
\maketitle
	
\begin{abstract}
We prove new complexity results for computational problems in certain wreath products of groups and (as an application)
for free solvable groups. For a finitely generated group we study the so-called power word problem (does a given expression $u_1^{k_1} \ldots u_d^{k_d}$,
where $u_1, \ldots, u_d$ are words over the group generators and $k_1, \ldots, k_d$ are binary encoded integers, evaluate to the group identity?)
and knapsack problem (does a given equation $u_1^{x_1} \ldots u_d^{x_d} = v$, where $u_1, \ldots, u_d,v$ are words over the group generators
and $x_1,\ldots,x_d$ are variables, have a solution in the natural numbers). We prove that the power word problem for wreath products
of the form $G \wr \Z$ with $G$ nilpotent and iterated wreath products of free abelian groups belongs to $\TC^0$. As an application of the latter,
the power word problem for free solvable groups is in $\TC^0$. On the other hand we show that for wreath products $G \wr \Z$, where $G$ is a
so called uniformly strongly efficiently non-solvable group (which form a large subclass of non-solvable groups), 
the power word problem is $\coNP$-hard. For the knapsack problem we show $\NP$-completeness for iterated
wreath products of free abelian groups and hence free solvable groups. Moreover, the knapsack problem for every wreath product
$G \wr \Z$, where $G$ is uniformly efficiently non-solvable, is $\Sigma_2^p$-hard.
\end{abstract}

 \section{Introduction}
  
Since its very beginning, the area of combinatorial group theory \cite{LySch77} is tightly connected to algorithmic questions. 
The word problem for finitely generated (f.g.~for short) groups lies at
the heart of theoretical computer science itself. Dehn \cite{Dehn11} proved its decidability for certain surface groups (before the notion of decidability
was formalized). Magnus \cite{Mag32} extended this result to all one-relator groups. 
After the work of Magnus it took more than 20 years before Novikov \cite{Nov58} and Boone 
\cite{Boo59} proved the existence of finitely presented groups with
an undecidable word problem (Turing tried to prove the existence of such groups but could only provide finitely presented cancellative monoids
with an undecidable word problem).

Since the above mentioned pioneering work, the area of algorithmic group theory has been extended in many different directions. More 
general algorithmic problems have been studied and also the computational complexity of group theoretic problems has been investigated.
In this paper, we focus on the decidability/complexity of two specific problems in group theory that have received considerable attention in recent years:
the knapsack problem and the power word problem. 

\subparagraph{Knapsack problems.} There exist several variants of the classical knapsack problem over the integers \cite{Karp72}.
In the variant that is particularly relevant for this paper, it is asked whether a linear equation 
$x_1 \cdot a_1 + \cdots + x_d \cdot a_d = b$, with $a_1, \ldots, a_d,b \in \Z$, has a solution $(x_1, \ldots, x_d) \in \N^d$.
A proof for the $\NP$-completeness of this problem for binary encoded integers $a_1, \ldots, a_d,b$ can be found in 
\cite{Haa11}. In contrast, if the numbers $a_i,b$ are given in unary notation then the problem falls down into the circuit
complexity class $\mathsf{TC}^0$ \cite{ElberfeldJT11}. In the course of a 
systematic investigation of classical commutative discrete optimization problems 
in non-commutative group theory, Myasnikov, Nikolaev, and Ushakov \cite{MyNiUs14} generalized the above definition of knapsack
to any f.g.~group $G$: The input for the knapsack problem for $G$ ($\KP(G)$ for short) is an equation of the form $g_1^{x_1} \cdots g_d^{x_d} = h$
for group elements $g_1, \ldots, g_d, h \in G$ (specified
by finite words over the generators of $G$) and pairwise different variables $x_1, \ldots, x_d$ that take values in $\N$
and it is asked whether this equation has a solution
(in the main part of the paper, we formulate this problem in a slightly more 
general but equivalent way). In this form, $\KP(\Z)$ is exactly the above knapsack problem for unary encoded
integers studied in \cite{ElberfeldJT11} (a unary encoded integer can be viewed as a word over a generating set
$\{t,t^{-1}\}$ of $\Z$). For the case where $g_1, \ldots, g_d, h$ are commuting matrices over an algebraic number field, 
the knapsack problem has been studied in \cite{BabaiBCIL96}.
Let us emphasize that we are looking for solutions of knapsack equations in the natural numbers. 
One might also consider the variant, where the variables $x_1, \ldots, x_d$ take values in $\Z$.
This latter version can be easily reduced to our knapsack version (with solutions in $\N$), 
but we are not aware of a reduction in the opposite direction.\footnote{Note that the problem whether a given system of linear equations
has a solution in $\N$ is $\NP$-complete, whereas the problem can be solved in polynomial time (using the Smith normal form) if we ask
for a solution in $\Z$. In other words, if we consider the knapsack problem for $\Z^n$ with $n$ part of the input, then looking for solutions
in $\N$ seems to be more difficult than looking for solutions in $\Z$.}
 Let us also mention that the knapsack problem is a special case of the more general rational subset membership problem \cite{LohreySteinbergZetzsche2015a}.

We also consider a generalization of $\KP(G)$: An exponent equation is an equation of the form 
$g_1^{x_1} \cdots g_d^{x_d} = h$ as in the specification of $\KP(G)$, except that the variables $x_1, \ldots, x_d$
are not required to be pairwise different.
{\em Solvability of exponent equations} for $G$ ($\ExpEq(G)$ for short) is the problem
where the input is a conjunction of exponent equations (possibly with shared variables) and the question is whether there is a joint
solution for these equations in the natural numbers. 

Let us give a brief survey over the results that were obtained for the knapsack problem in \cite{MyNiUs14}  and successive papers:
\begin{itemize}
\item Knapsack can be solved in polynomial
time for every hyperbolic group \cite{MyNiUs14}. Some extensions of this result can be found 
in \cite{FrenkelNU15,Loh19hyp}.
\item There are nilpotent groups of class $2$ for which knapsack is undecidable. Examples
are direct products of sufficiently many copies of the discrete Heisenberg group $H_3(\mathbb{Z})$ \cite{KoenigLohreyZetzsche2015a},
and free nilpotent groups of class $2$ and sufficiently high rank \cite{MiTr17}. In contrast, 
knapsack for $H_3(\mathbb{Z})$ is decidable \cite{KoenigLohreyZetzsche2015a}. It follows  that decidability
of knapsack is not preserved under direct products.
\item Knapsack is decidable for every co-context-free group~\cite{KoenigLohreyZetzsche2015a}, i.e., groups where the set of all words over the generators that do not represent the identity
is a context-free language. Lehnert and Schweitzer \cite{LehSch07} have shown that the Higman-Thompson groups are
co-context-free.
\item Knapsack belongs to $\NP$ for all virtually special groups (finite extensions of subgroups
of graph groups) \cite{LohreyZetzsche2016a}. The class of virtually special groups is very rich. It contains all Coxeter groups, 
one-relator groups with torsion, fully residually free groups, and  fundamental groups of hyperbolic 3-manifolds.
For graph groups (also known as right-angled Artin groups) a complete classification of the complexity
of knapsack was obtained in \cite{LohreyZ18}: If the underlying graph contains an induced path or cycle on 4 nodes, then knapsack
is $\NP$-complete; in all other cases knapsack can be solved in polynomial time (even in {\sf LogCFL}).
\item Knapsack is  $\NP$-complete for every wreath products $A \wr \mathbb{Z}$ with $A \neq 1$ f.g. abelian \cite{GanardiKLZ18}
(wreath products are formally defined in Section~\ref{sec-wreath}).
\item Decidability of knapsack is preserved under finite extensions, HNN-exten\-sions over finite associated subgroups 
and amalgamated free products over finite subgroups \cite{LohreyZetzsche2016a}.
\end{itemize}
For a knapsack equation $g_1^{x_1} \cdots g_d^{x_d} = h$ we may consider the set of all solutions
$\{ (n_1, \ldots, n_d) \in \mathbb{N}^d \mid g_1^{n_1} \cdots g_d^{n_d} = g \text{ in } G\}$. In the papers
\cite{Loh19hyp,KoenigLohreyZetzsche2015a,LohreyZ18} it turned out that in many groups the solution
set of every knapsack equation is a {\em semilinear set} (see Section~\ref{sec-prel} for a definition). 
We say that a group is {\em knapsack-semilinear} if for every knapsack equation the set of all solutions is semilinear
and a semilinear representation can be computed effectively (the same holds then also for exponent equations).
Note that in any group $G$ the set of solutions on an equation $g^x = h$ is periodic and hence semilinear. 
This result generalizes to solution sets of knapsack instances of the form $g_1^x g_2^y = h$ (see Lemma~\ref{lemma-2dim}), but there
are examples of knapsack equations with three variables where solution sets (in certain groups) are not semilinear.
Examples of knapsack-semilinear groups are graph groups \cite{LohreyZ18} (which include free groups and free abelian groups),
hyperbolic groups \cite{Loh19hyp}, and co-context free groups 
\cite{KoenigLohreyZetzsche2015a}.\footnote{Knapsack-semilinearity of co-context free groups is not stated in 
\cite{KoenigLohreyZetzsche2015a} but follows immediately from the proof for the decidability of knapsack.}
Moreover, the class of knapsack-semilinear groups is closed under finite extensions, graph products, 
amalgamated free products with finite amalgamated subgroups, HNN-extensions with finite associated subgroups
(see \cite{FL19} for these closure properties)
and wreath products  \cite{GanardiKLZ18}.

\subparagraph{Power word problems.}
In the power word problem for a f.g.~group $G$ ($\PowWP(G)$ for short)
the input consists of an expression $u_1^{n_1} u_2^{n_2} \cdots u_d^{n_d}$,
where $u_1, \ldots, u_d$ are words over the group generators and $n_1, \ldots, n_d$ are binary encoded integers. The problem is then to decide whether the expression $u_1^{n_1}u_2^{n_2}\cdots u_d^{n_d}$
evaluates to the identity in $G$. The power word problem
arises very naturally in the context of the knapsack problem: it allows us to verify a proposed solution for a knapsack equation with binary encoded numbers.
The power word problem has been first studied in \cite{LoWe19}, where it was shown that the power word problem for f.g.~free groups has the same complexity as the word problem and hence can be solved in logarithmic space. Other groups with easy power word problems are f.g.~nilpotent groups and wreath products $A \wr \mathbb{Z}$ with $A$ f.g.~abelian \cite{LoWe19}. In contrast it is shown in \cite{LoWe19}
that the power word problem for wreath products $G \wr \mathbb{Z}$, where $G$ is either finite non-solvable or f.g.~free, is $\coNP$-complete. 
Implicitly, the power word problem appeared also in the work of Ge \cite{Ge93}, where it was shown that one can verify in polynomial time an identity 
 $\alpha_1^{n_1} \alpha_2^{n_2}\cdots \alpha_d^{n_d} = 1$, where the $\alpha_i$ are elements of an algebraic
 number field and the $n_i$ are binary encoded integers. Let us also remark that the power word problem is a special case of the compressed
 word problem \cite{Loh14}, which asks whether a grammar-compressed word over the group generators evaluates to the group identity.

\subparagraph{Main results.}
In this paper, we are mainly interested in the problems $\PowWP(G)$, $\KP(G)$ and $\ExpEq(G)$ for the case where 
$G$ is a wreath product.
We start with the following result:

\begin{theorem} \label{thm:nilpotent-power}
  Let $G$ be a f.g.~nilpotent group. Then $\PowWP(G \wr \mathbb{Z})$ is in $\TC^0$.
  \end{theorem}
Theorem~\ref{thm:nilpotent-power} generalizes the above mentioned result from \cite{LoWe19}
(for $G$ abelian) in a nontrivial way. Our proof analyzes periodic infinite words over a nilpotent group $G$. Roughly speaking, we show that one can check in $\TC^0$,
whether a given  list of such periodic infinite words pointwise multiplies to the identity of $G$. We believe that this is a result of independent
interest.  We use this result also in the proof of the following theorem:

\begin{theorem} \label{thm:nilpotent-KP}
  Let $G$ be a finite nontrivial nilpotent group. Then $\KP(G \wr \mathbb{Z})$ is $\NP$-complete.
\end{theorem}
Next, we consider iterated wreath products.
Fix a number $r \geq 1$ and let us define the iterated wreath products
$W_{0,r} =\Z^r$ and $W_{m+1,r}=\Z^r \wr W_{m,r}$.
By a famous result of Magnus~\cite{Mag39} the free solvable group $S_{m,r}$ of derived length $r$ and rank $m$
embeds into $W_{m,r}$.  Our main results for these groups are:

\begin{theorem}\label{thm:iterated-pwp}
$\PowWP(W_{m,r})$ and hence $\PowWP(S_{m,r})$ belong to $\TC^0$ for all $m \geq 0$, $r \geq 1$.
\end{theorem}
It was only recently shown in \cite{MVW19} that the word problem
(as well as the conjugacy problem) for every
free solvable group belongs to $\TC^0$. Theorem~\ref{thm:iterated-pwp} generalizes this result
(at least the part on the word problem).

\begin{theorem}\label{thm:iterated-kp}
$\ExpEq(W_{m,r})$ and hence $\ExpEq(S_{m,r})$ are $\NP$-complete  for all $m \geq 0$, $r \geq 1$. 
\end{theorem}
For the proof of Theorem~\ref{thm:iterated-kp} we show that if a given knapsack equation over
$W_{m,r}$ has a solution then it has a solution where all numbers are exponentially bounded in the length of the knapsack instance. 
Theorem~\ref{thm:iterated-kp} then follows easily from Theorem~\ref{thm:iterated-pwp}.
For some other algorithmic results for free solvable groups see \cite{MRUV10}.

Finally, we prove a new hardness results for the power word problem and knapsack problem.
For this we make use so-called {\em uniformly strongly efficiently non-solvable} groups (uniformly SENS groups) that were recently defined
in \cite{BFLW20}. Roughly speaking, a group $G$ is uniformly 
SENS if there exist nontrivial nested commutators of arbitrary depth that moreover, are efficiently computable in a certain sense
(see Section~\ref{sec-SENS} for the precise definition). The essence of these groups is that they allow to carry out Barrington's argument showing the $\NC^1$-hardness of the word problem for a finite solvable group \cite{Bar89}. 
We prove the following:

\begin{theorem} \label{cor:wreath-SENS-power}
   Let the f.g.~group $G = \langle \Sigma\rangle$ be uniformly SENS. Then, $\PowWP(G \wr \mathbb{Z})$ is $\coNP$-hard.
\end{theorem}
This result generalizes a result from \cite{LoWe19} saying that $\PowWP(G \wr \mathbb{Z})$ is $\coNP$-hard for the case that $G$
is f.g.~free or finite non-solvable.

\begin{theorem} \label{cor:wreath-SENS}
   Let the f.g.~group $G = \langle \Sigma\rangle$ be uniformly SENS. Then, $\KP(G \wr \mathbb{Z})$ is $\Sigma^p_2$-hard.
\end{theorem}
Recall that for every nontrivial group $G$, $\KP(G \wr \Z)$ is $\NP$-hard \cite{GanardiKLZ18}.

In the main part we also state several corollaries of Theorem~\ref{cor:wreath-SENS-power} and \ref{cor:wreath-SENS}.
For instance, we show that for the famous Thompson's group $F$, $\PowWP(F)$ is $\coNP$-complete and 
$\KP(F)$ is $\Sigma^p_2$-hard.

\section{Preliminaries} \label{sec-prel}
  
 \subparagraph{Complexity theory.}
  
 We assume some knowledge in complexity theory; in particular the reader should be familiar
 with the classes $\Ptime$, $\NP$, and $\coNP$.  The class $\Sigma_2^p$ (second existential
 level of the polynomial time hierarchy) contains all languages $L \subseteq \Sigma^*$ for which there exists a polynomial $p$
 and a language $K \subseteq \Sigma^* \# \{0,1\}^* \# \{0,1\}^*$ in $\Ptime$ (for a symbol $\# \notin \Sigma \cup \{0,1\}$) such that
 $x \in L$ if and only if $\exists y \in \{0,1\}^{\leq p(|x|)} \forall z \in \{0,1\}^{\leq p(|x|)} : x \# y \# z \in K$.
 
 The class $\TC^0$ contains all problems that can be solved by a family of threshold circuits of 
polynomial size and constant depth. In this paper, $\TC^0$ will always
refer to the $\mathsf{DLOGTIME}$-uniform version of $\TC^0$.
A precise definition is not needed for our work; see \cite{Vol99} for details. All we need is that the following arithmetic 
operations on binary encoded integers belong to $\TC^0$: iterated addition and multiplication  (i.e., addition and multiplication 
of $n$ many $n$-bit numbers) and division with remainder.

For languages (or computational problems) $A, B_1,\ldots,B_k \subseteq \{0,1\}^*$ we write 
$A \in \TC^0(B_1,\ldots,B_k)$ ($A$ is $\TC^0$-Turing-reducible to $B_1,\ldots,B_k$) if $A$ can be solved by a family of threshold circuits of polynomial size and constant depth that in addition may also use oracle gates for the languages $B_1,\ldots,B_k$
 (an oracle gate for $B_i$ yields the output $1$ if and only if the string of input bits belongs to $B_i$).

\subparagraph{Arithmetic progressions.}

An \emph{arithmetic progression} is a tuple $\bm{p} = (a + pi)_{0 \leq i \leq k}$ for some
$a,p,k \in \N$ with $p \neq 0$.  We call $a$ the {\em offset}, $p$ the {\em period} and $k+1$ the {\em length} 
of $P$. The {\em support} of $\bm{p}$ is $\supp(\bm{p}) = \{ a + pi \mid 0 \leq i \leq k\}$.
In computational problems we will represent the arithmetic progression $\bm{p}$ by the triple
$(a,p,k+1)$, where the offset $a$ and the length $k+1$ are represented in binary notation
whereas the period $p$ is represented in unary notation (i.e., as the string $\$^p$ for some 
special symbol $\$$).

\subparagraph{Intervals.}
A subset $B$ in a linear order $(A,\le)$ is an {\em interval}
if $a_1 \le a_2 \le a_3$ and $a_1,a_3 \in B$ implies $a_2 \in B$.

\begin{lemma}
	\label{lem:interval-types}
	Let $(A,\le)$ be a linear order, let $\Omega$ be a finite set of colors
	and let $\beta \colon A \to 2^\Omega$ be a mapping such that
	$\{ a \in A \mid \omega \in \beta(a) \}$ is an interval for each $\omega \in \Omega$.
	Then there exists a partition of $A$ into at most $\O(|\Omega|)$ intervals $A_1, \dots, A_k$
	such that $|\beta(A_i)| = 1$ for all $1 \le i \le k$.
	Furthermore, if $A = [0,n]$ and each interval $\{ a \in A \mid \omega \in \beta(a) \}$
	is given by its endpoints (in binary encoding)
	we can compute the endpoints of the intervals $A_1, \dots, A_k$
	in $\TC^0$.
\end{lemma}

\begin{proof}
We prove that there exists such a partition with at most $2|\Omega|+1$ many intervals
by induction on $|\Omega|$. The case $|\Omega| = 0$ is clear.
Now let $\Omega = \Omega' \cup \{\omega\}$ where $\omega \notin \Omega'$
and let $\beta'(a) = \beta(a) \cap \Omega'$, which still satisfies the condition from the lemma.
By induction we obtain a partition of $A$ into at most $2|\Omega|-1$
intervals $A_1, \dots, A_k$ such that $|\beta'(A_i)| = 1$ for all $1 \le i \le k$.
Now consider the interval $A_0 = \{ a \in A \mid \omega \in \beta(a) \}$.
If $A_i$ is contained in $A_0$ or $A_i$ is disjoint from $A_0$
then $|\beta(A_i)| = |\beta'(A_i)| = 1$.
Otherwise $A_i$ can be partitioned into the intervals $A_i \cap A_0$ and $A_i \setminus A_0$,
which also satisfy $|\beta(A_i \cap A_0)| = |\beta(A_i \setminus A_0)| = 1$.
Since there are at most two such intervals $A_i$ whose symmetric difference with $A_0$ is non-empty,
at most two intervals are added in total.

For the $\TC^0$-statement we take a different approach.
Let $P$ be the set of all (at most $2|\Omega|$) endpoints
of the intervals $\{ a \in A \mid \omega \in \beta(a) \}$ for $\omega \in \Omega$
together with the minimum $0$ and the maximum $n$.
We sort $P$ in $\TC^0$ \cite{ChandraSV84}, say $P = \{a_1, \dots, a_m\}$ with $a_1 < a_2 < \dots < a_m$, and
define the partition
consisting of all singletons $\{a_i\}$ for $1 \le i \le m$
and all ``gap'' intervals $[a_{i-1}+1,a_i-1]$ for $2 \le i \le m$ with $a_{i-1}+1 \leq a_i-1$.
We clearly have $|\beta(\{a_i\})| = 1$.
Now consider $a,b \in [a_{i-1}+1,a_i-1]$ with $a < b$
and assume that $\beta(a) \neq \beta(b)$, i.e. there exists $\omega \in \Omega$
with $\omega \in \beta(a) \setminus \beta(b)$ (or $\omega \in \beta(b) \setminus \beta(a)$).
Let $d$ be the right endpoints of $\{ c \in A \mid \omega \in \beta(c) \}$,
which must satisfy $a \le d < b$.
But then $a_{i-1}+1 \le a \le d < b \le a_i-1$,
and therefore $a_{i-1} < d < a_i$, which is a contradiction.
\end{proof}

 \subparagraph{Semilinear sets.} 

Fix a dimension $d \ge 1$.
All vectors will be column vectors. 
For a vector $\bm{v}=(v_1, \dots , v_d)^{\mathsf{T}}\in \Z^d$ we define its norm
$\norm{\bm{v}}:=\max \{ |v_i| \mid 1 \leq i \leq d \}$ and for a 
matrix $M \in \Z^{c \times d}$ with entries $m_{i,j}$ ($1 \le i \le c$, $1 \le j \le d$) we define the norm
$\norm{M} = \max \{|m_{i,j}| \mid 1 \le i \le c, \, 1 \leq j \leq d \}$. Finally, for a finite set of vectors $A \subseteq \N^d$ let
$\norm{A} = \max \{ \norm{\bm{a}} \mid \bm{a} \in A \}$.

We extend the operations of vector addition and multiplication of a vector by a matrix to sets of vectors in the obvious way.
A linear subset of $\N^d$ is a set of the form
\[
L= L(\bm{b},P) := \bm{b} + P \cdot \N^k
\]
where $\bm{b} \in \N^d$ and $P \in \N^{d \times k}$.
We call a set $S\subseteq \N^d$ \emph{semilinear}, if it is a finite union of linear sets. 
Semilinear sets play a very important role in many areas of computer science and mathematics, e.g. in automata theory and logic.
It is known that the class of semilinear sets is closed under Boolean operations and that the semilinear sets are exactly
the Presburger definable sets (i.e., those sets that are definable in the structure $(\mathbb{N},+)$).

For a semilinear set $S=\bigcup_{i=1}^k L(\bm{b}_i,P_i)$, we call the tuple
$(\bm{b}_1, P_1, \ldots, \bm{b}_k, P_k)$
a \emph{semilinear representation} of $S$.
The magnitude of the semilinear representation $(\bm{b}_1, P_1, \ldots, \bm{b}_k, P_k)$
is $\max \{ \norm{\bm{b}_1}, \norm{P_1} \ldots, \norm{\bm{b}_k}, \norm{P_k}\}$.
The \emph{magnitude} $\|S\|$ of a semilinear 
set $S$ is the minimal magnitude of all semilinear representations for $S$.

\begin{lemma}[\cite{HaaseZ19}]
\label{lem:cap}
If $M_1, \dots, M_k \subseteq \N^d$ are semilinear sets with $\|M_i\| \le s$
then 
\[ \|\bigcap_{i=1}^k M_i\| \le (s \cdot k \cdot d + 1)^{\mathcal{O}(k \cdot d)}.
\]
\end{lemma}
In the context of knapsack problems (which we will introduce in the next section), we will consider 
semilinear subsets as sets of mappings $\nu  \colon  \{x_1, \ldots, x_d\} \to \N$ for a finite set of variables
$X = \{x_1, \ldots, x_d\}$. Such a mapping $f$ can be identified with the vector $(\nu(x_1), \ldots, \nu(x_d))^{\mathsf{T}}$.
This allows to use all vector operations (e.g. addition and scalar multiplication) on the set
$\N^X$ of all mappings from $X$ to $\N$.

  \section{Groups} \label{sec-groups}
  
 We assume that the reader is familiar with the basics of group theory.
Let $G$ be a group. We always write $1$ for the group identity element.
For $g,h \in G$ we write $[g,h] := g^{-1} h^{-1}gh$ for the commutator of $g$ and $h$
and $g^h$ for $h^{-1} g h$.  For subgroups $A,B$ of $G$ we write $[A,B]$ for the subgroup 
generated by all commutators $[a,b]$ with $a \in A$ and $b \in B$.
The order of an element $g \in G$ is the smallest number $z > 0$ with $g^z = 1$ and $\infty$
if such a $z$ does not exist. The group $G$ is torsion-free, if every $g \in G \setminus \{1\}$ has infinite order.

We say that $G$ is finitely generated (f.g.) if there is a finite subset $\Sigma \subseteq G$
such that every element of $G$ can be written as a product of elements from $\Sigma$; such a $\Sigma$ is called
a finite generating set for $G$. We also write $G = \langle \Sigma \rangle$.
We then have a canonical morphism $h \colon \Sigma^* \to G$ that maps a word over $\Sigma$ to its product in $G$. 
 If $h(w)=1$ we also say that $w=1$ in $G$. For $g \in G$ we write $|g|$ for the length of a shortest word 
  $w \in \Sigma^*$ such that $h(w) = g$. This notation depends on the generating set $\Sigma$.
 We always assume that the generating set $\Sigma$ is symmetric in the sense that $a \in \Sigma$ implies 
 $a^{-1} \in \Sigma$. Then, we can define 
on $\Sigma^*$ a natural involution $\cdot^{-1}$ by $(a_1 a_2 \cdots a_n)^{-1} = a_n^{-1} \cdots a_2^{-1} a_1^{-1}$
for $a_1, a_2, \ldots, a_n \in \Sigma$. This allows to use the notations $[g,h] = g^{-1} h^{-1}gh$ and 
$g^h = h^{-1} g h$ also in case $g,h \in \Sigma^*$.
In the following, when we say that we want to compute a homomorphism $h  \colon  G_1 = \langle \Sigma_1 \rangle \to G_2 = \langle \Sigma_2 \rangle$,
 we always mean that we compute the images $h(a)$ for $a \in \Sigma_1$.
   
 A group $G$ is called {\em orderable} if there exists a linear order $\leq$ on $G$ such that
$g \leq h$ implies $xgy \leq xhy$ for all $g,h,x,y \in G$ \cite{neumann1949,mura1977}.
 Every orderable group is torsion-free (this follows directly from the definition) and has the unique roots property \cite{rolfsen2014}, i.e.,
 $g^n = h^n$ implies $g = h$. The are numerous examples of orderable groups: for instance, 
 torsion-free nilpotent groups, right-angled Artin groups, and diagram groups are all orderable.
 
 \subsection{Commensurable elements} \label{sec-commens} 
 
 Two elements $g,h \in G$ in a group $G$ are called {\em commensurable} if $g^x = h^y$
for some $x,y \in \Z \setminus \{0\}$.
This defines an equivalence relation on $G$,
in which the elements with finite order form an equivalence class.
By \cite[Corollary~1.2]{neumann1949} commensurable elements in an orderable group commute.

\begin{lemma} \label{lemma-commensurable-cyclic}
Let $G$ be an orderable group and let $U \subseteq G$ be a finite set of pairwise commensurable elements.
Then $\langle U \rangle$ is a cyclic subgroup of $G$.
\end{lemma}

\begin{proof}
Recall that $G$ is torsion-free and has the unique roots property.
We prove the lemma by induction on the size of $U$. The case $|U|=1$ is obvious.
Now assume that $|U| > 1$. By the above mentioned result from \cite{neumann1949} $\langle U \rangle$ is abelian.
Choose arbitrary elements $g,h \in U$ with $g \neq h$.
Since $g$ and $h$ are commensurable, there exist $p,q \in \Z \setminus \{0\}$ with 
$g^p = h^q$. Since $G$ has the unique roots property, we can assume that $\mathrm{gcd}(p,q) = 1$.
Hence, there exist $k,\ell \in \Z$ with $1 = k p + \ell q$. Consider
the group element $a = g^\ell h^k$. We then have $g = g^{k p + \ell q} = g^{\ell q} h^{k q} = a^q$ and similarly
$h = a^p$. We therefore have $\langle g,h\rangle = \langle a \rangle$. Note that $a \neq 1$ since  $\langle g,h\rangle \neq 1$.

We next claim that every $b \in U \setminus \{g,h\}$ is commensurable to $a$.
Since $g$ (resp., $h$) is commensurable to $b$, there exist 
$r,s,t,u \in \Z \setminus \{0\}$ with $g^r = b^s$ and $h^t = b^u$. 
We obtain $a^{rt} = g^{\ell r t} h^{k r t} = b^{\ell s t + k r u}$.
Finally, note that since $rt \neq 0$ and $G$ is torsion-free, we must have $\ell s t + k r u \neq 0$.

We have shown that $V = (U \setminus \{g,h\}) \cup \{a\}$ consists of pairwise commensurable elements. By induction,
 $\langle V \rangle$ is cyclic. Moreover, $\langle g,h\rangle = \langle a \rangle$ implies that 
 $\langle U \rangle = \langle V \rangle$, which proves the lemma.
\end{proof}

  \subsection{Wreath products} \label{sec-wreath}
  
  Let $G$ and $H$ be groups. Consider the direct sum $K = \bigoplus_{h
  \in H} G_h$, where $G_h$ is a copy of $G$. We view $K$ as the set $G^{(H)}$ of
all mappings $f\colon H\to G$ such that $\supp(f) := \{h\in H \mid f(h)\ne
1\}$ is finite, together with pointwise multiplication as the group
operation.  The set $\supp(f)\subseteq H$ is called the
\emph{support} of $f$. The group $H$ has a natural left action on
$G^{(H)}$ given by $h f(a) = f(h^{-1}a)$, where $f \in G^{(H)}$ and
$h, a \in H$.  The corresponding semidirect product $G^{(H)} \rtimes
H$ is the (restricted) \emph{wreath product} $G \wr H$.  In other words:
\begin{itemize}
\item
Elements of $G \wr H$ are pairs $(f,h)$, where $h \in H$ and
$f \in G^{(H)}$.
\item
The multiplication in $G \wr H$ is defined as follows:
Let $(f_1,h_1), (f_2,h_2) \in G \wr H$. Then
$(f_1,h_1)(f_2,h_2) = (f, h_1h_2)$, where
$f(a) = f_1(a)f_2(h_1^{-1}a)$.
\end{itemize}
There are canonical mappings
\begin{itemize}
\item $\sigma \colon G \wr H \to H$ with $\sigma(f,h) = h$ and  
\item  $\tau \colon G \wr H \to G^{(H)}$ with $\tau(f,h) = f$
\end{itemize}
In other words: $g = (\tau(g), \sigma(g))$ for $g \in G \wr H$.
Note that $\sigma$ is a homomorphism whereas $\tau$ is in general not a homomorphism.
Throughout this paper, the letters $\sigma$ and $\tau$ will have the above meaning, which of course
depends on the underlying wreath product $G \wr H$, but the latter will be always clear from the context.

The following intuition might be helpful:
An element $(f,h) \in G\wr H$ can be thought of
as a finite multiset of elements of $G \setminus\{1_G\}$ that are sitting at certain
elements of $H$ (the mapping $f$) together with the distinguished
element $h \in H$, which can be thought of as a cursor
moving in $H$.
If we want to compute the product $(f_1,h_1) (f_2,h_2)$, we do this
as follows: First, we shift the finite collection of $G$-elements that
corresponds to the mapping $f_2$ by $h_1$: If the element $g \in G\setminus\{1_G\}$ is
sitting at $a \in H$ (i.e., $f_2(a)=g$), then we remove $g$ from $a$ and
put it to the new location $h_1a \in H$. This new collection
corresponds to the mapping $f'_2 \colon  a \mapsto f_2(h_1^{-1}a)$.
After this shift, we multiply the two collections of $G$-elements
pointwise: If in $a \in H$ the elements $g_1$ and $g_2$ are sitting
(i.e., $f_1(a)=g_1$ and $f'_2(a)=g_2$), then we put the product
$g_1g_2$ into the location $a$. Finally, the new distinguished
$H$-element (the new cursor position) becomes $h_1 h_2$.

Clearly, $H$ is a subgroup of $G \wr H$. But also $G$ is a subgroup of $G \wr H$. We can identify 
$G$ with the set of all mappings $f \in G^{(H)}$ such that $\supp(f) \subseteq \{1\}$. This copy of $G$ together 
with $H$ generates $G \wr H$.   In particular, if $G = \langle \Sigma \rangle$ and $H = \langle \Gamma \rangle$ 
with $\Sigma \cap \Gamma = \emptyset$ then $G \wr H$ is generated by $\Sigma \cup \Gamma$. 
In this situation, we will also apply the above mappings $\sigma$ and $\tau$ to words over 
$\Sigma \cup \Gamma$.  We will need the following embedding result: 
  
\begin{lemma} \label{lemma-wreath-embedding}
Given a unary encoded number $d$, one can compute in logspace an embedding of $G^d \wr \mathbb{Z}$ into $G \wr \mathbb{Z}$.
\end{lemma} 
  
\begin{proof}
Let $G = \langle \Gamma \rangle$ and let $\Gamma_i$ ($0 \leq i \leq d-1$) be pairwise disjoint copies of $\Gamma$, each of which
generates a copy of $G$.
For $G^d \wr \mathbb{Z}$ we take the generating set $\{t,t^{-1}\} \cup \bigcup_{i=0}^{d-1} \Gamma_i$, where $t$ 
generates the right factor $\mathbb{Z}$.
We then obtain an embedding $h  \colon  G^d \wr \mathbb{Z} \to G \wr \mathbb{Z}$ by:
\begin{itemize}
\item $h(t) = t^d$ and $h(t^{-1}) = t^{-d}$,
\item $h(a) = t^i a t^{-i}$ for $a \in \Gamma_i$.
\end{itemize}
This proves the lemma.
\end{proof}
In \cite{MRUV10} it was shown that the word problem of a wreath product $G \wr H$ is $\TC^0$-reducible
to the word problems for $G$ and $H$. Let us briefly sketch the argument. Assume that 
$G = \langle \Sigma \rangle$ and $H = \langle \Gamma \rangle$.
 Given a word $w \in (\Sigma \cup \Gamma)^*$ one has to check whether $\sigma(w)=1$ in $H$ and $\tau(w)(h) = 1$ in $H$
 for all $h$ in the support of $\tau(w)$. One can compute in $\TC^0$ the word $\sigma(w)$ 
 by projecting $w$ onto the alphabet $\Gamma$. Moreover, one can enumerate the support of $\tau(w)$ 
 by going over all prefixes of $w$ and checking which $\sigma$-values are the same. Similarly, one produces
 for a given $h \in \supp(\tau(w))$ a word over $\Sigma$ that represents $\tau(w)(h)$.

\begin{lemma}
\label{lem:tau}
For $g_1, \dots, g_k \in G \wr H$ we have
$\tau(g_1 \cdots g_k) = \prod_{i=1}^k \tau(\sigma(g_1 \cdots g_{i-1}) \, g_i)$.
\end{lemma}

\begin{proof}
By definition of the wreath product we have (for better readability we write $\circ$ for the multiplication in $G$):
\[
	\tau(g_1 g_2)(h) = \tau(g_1)(h) \circ \tau(g_2)(\sigma(g_1)^{-1} h) = \tau(g_1)(h) \circ \tau(\sigma(g_1) \, g_2)(h)
\]
for all $h \in H$ and therefore $\tau(g_1 g_2) = \tau(g_1) \circ \tau(\sigma(g_1) \, g_2)$, which is the case $k = 2$.
The general statement follows by induction.
\end{proof}
Finally, we need the following result from  \cite{longobardi1998}:
\begin{theorem}[\cite{longobardi1998}] \label{thm-wreath-orderable}
If $G$ and $H$ are orderable then also $G \wr H$ is orderable.\footnote{This holds only for the restricted
wreath product; which is the wreath product construction we are dealing with.}
\end{theorem}

 \subsection{Knapsack problem} \label{sec-knapsack}
 
 Let $G = \langle \Sigma \rangle$ be a f.g.~group.
Moreover, let $X$ be a set of formal variables that take values
from $\mathbb{N}$. For a subset $Y\subseteq X$, we use $\mathbb{N}^Y$ to denote the set of 
maps  $\nu \colon Y \to \mathbb{N}$, which we call \emph{valuations}.
For valuations $\nu \in\mathbb{N}^Y$ and $\mu \in\mathbb{N}^Z$ such that $Y \subseteq Z$
we say that $\nu$ extends $\mu$ (or $\mu$  restricts to $\nu$) if $\nu(x) = \mu(x)$ for all $x \in Y$.

An \emph{exponent expression} over $G$ is an expression of the form $E = v_0 u_1^{x_1} v_1 u_2^{x_2} v_2  \cdots u_d^{x_d} v_d$
with $d \geq 1$, words $v_0,\ldots, v_d \in \Sigma^*$, non-empty words $u_1, \ldots, u_d \in \Sigma^*$, and variables $x_1, \ldots, x_d$.
Here, we allow $x_i = x_j$ for $i \neq j$. If every variable $x_i$ occurs at most once, then $E$ is called a \emph{knapsack expression}. 
Let $X = \{ x_1, \ldots, x_d \}$ be the set of variables that occur in $E$.
For a homomorphism $h  \colon  G \to G' = \langle \Sigma'\rangle$ (that is specified by a mapping from $\Sigma$ to $(\Sigma'\cup \Sigma'^{-1})^*$),
we denote with $h(E)$ the exponent expression 
$h(u_1)^{x_1} h(v_1) h(u_2)^{x_2} h(v_2)  \cdots h(u_d)^{x_d} h(v_d)$.
 For a valuation $\nu\in\mathbb{N}^Y$ such that $X \subseteq Y$ (in which case we also say
 that $\nu$ is a valuation for $E$), we define 
 $\nu(E) = u_1^{\nu(x_1)} v_1 u_2^{\nu(x_2)} v_2  \cdots u_d^{\nu(x_d)} v_d \in \Sigma^*$.
We say that $\nu$ is a \emph{$G$-solution} for $E$ if $\nu(E)=1$ in $G$.
With $\Sol_G(E)$ we denote the set of all $G$-solutions $\nu \in \mathbb{N}^{X}$ of $E$.   
The \emph{length} of $E$ is defined as $|E| = \sum_{i=1}^d |u_i|+|v_i|$.
We define {\em solvability of exponent equations over $G$}, $\ExpEq(G)$ for short,  as the following decision problem:
\begin{description}
\item[Input] A finite list of exponent expressions $E_1,\ldots,E_n$ over $G$.
\item[Question] Is $\bigcap_{i=1}^n \Sol_G(E_i)$ non-empty?
\end{description}
The {\em knapsack problem for $G$}, $\KP(G)$ for short, is the following 
decision problem:
\begin{description}
\item[Input] A single knapsack expression $E$ over $G$.
\item[Question] Is $\Sol_G(E)$ non-empty?
\end{description}
It is easy to observe that the concrete choice of the generating set $\Sigma$ has no influence
on the decidability and complexity status of these problems.

 We could also restrict to knapsack expressions of the form $u_1^{x_1} u_2^{x_2}  \cdots u_d^{x_d} v$
 (but sometimes it will be convenient to allow nontrivial elements between the powers):
 for $E = v_0 u_1^{x_1} v_1 u_2^{x_2} v_2   \cdots u_d^{x_d} v_d$
and 
\[ E' = (v_0 u_1 v_0^{-1}) ^{x_1} (v_0 v_1 u_2 v_1^{-1} v_0^{-1})^{x_2} \cdots (v_0 \cdots v_{d-1} u_d v_{d-1}^{-1} \cdots v_0^{-1})^{x_d}
v_0 \cdots v_{d-1} v_d
 \]
we have $\Sol_G(E) = \Sol_G(E')$.

For the knapsack problem in wreath products the following result has been shown in \cite{GanardiKLZ18}:
\begin{theorem}[\cite{GanardiKLZ18}] \label{thm:NP-hard-KP}
For every nontrivial group $G$, $\KP(G \wr \mathbb{Z})$ is $\NP$-hard.
\end{theorem}

\subsection{Knapsack-semilinear groups} \label{sec-knapsack-semilinear}

The group $G$ is called {\em knapsack-semilinear} if for every knapsack expression $E$ over $\Sigma$,
the set $\Sol_G(E)$ is a semilinear set of vectors and a semilinear representation can be effectively computed from $E$.
Since semilinear sets are effectively closed under intersection, it follows that 
for every exponent expression $E$ over $\Sigma$,
the set $\Sol_G(E)$ is semilinear and a semilinear representation can be effectively computed from $E$.
Moreover, solvability of exponent equations is decidable for every knapsack-semilinear group.
As mentioned in the introduction, the class of knapsack-semilinear groups is very rich.
An example of a group $G$, where knapsack is decidable but solvability of exponent equations
is undecidable is the Heisenberg group $H_3(\mathbb{Z})$ 
(which consists of all upper triangular $(3 \times 3)$-matrices over the integers, where all diagonal entries
are $1$), see \cite{KoenigLohreyZetzsche2015a}. In particular, $H_3(\mathbb{Z})$ is not knapsack-semilinear.
In order to obtain a non-semilinear solution set, one needs a knapsack instance over $H_3(\mathbb{Z})$ with three variables.
In fact, for two variables we have the following simple fact:

\begin{lemma} \label{lemma-2dim}
Let $G$ be a group and $g_1, g_2, h \in G$ be elements.
\begin{enumerate}[(i)]
\item The solution set $S_1 = \{ (x, y) \in \Z^2 \mid g_1^x g_2^y = 1 \}$ is a subgroup of $\Z^2$.
If $G$ is torsion-free and $\{g_1,g_2\} \neq \{1\}$ then $S_1$ is cyclic.
\item The solution set $S = \{ (x, y) \in \Z^2 \mid g_1^x g_2^y = h \}$
is either empty or a coset $(a,b) + S_1$ of $S_1$ where $(a,b) \in S$ is any solution.
\end{enumerate}
\end{lemma}

\begin{proof}
Clearly $(0,0) \in S_1$, and if $g_1^x g_2^y = 1 = g_1^{x'} g_2^{y'}$
then also $g_1^{x-x'} g_2^{y-y'} = 1$. This shows the first part of statement (i).
Now assume that $G$ is torsion-free and that $g_1 \neq 1$
(the case $g_2 \neq 1$ is analogous).
If $(x,y),(x',y') \in S_1$ then $y'(x,y) - y(x',y') = (xy'-x'y,0) \in S_1$
and hence $g_1^{xy'-x'y}=1$.
Since $G$ is torsion-free this implies that $xy'-x'y=0$,
i.e. $(x,y)$ and $(x',y')$ are linearly dependent, since
$\det \big(\begin{smallmatrix}
x & x'\\
y & y'
\end{smallmatrix} \big)=0$.
This shows that $S_1$ is cyclic.

For (ii) let us assume that $S \neq \emptyset$ and take any solution $(a,b) \in S$,
i.e. $g_1^{a} g_2^{b} = h$.
We first show that $(a,b) + S_1 \subseteq S$.
Take any $(x,y) \in S_1$,
i.e. $g_1^{x} g_2^{y} = 1$.
Then we obtain
$g_1^{a + x} g_2^{b + y} = g_1^a g_1^x g_2^y g_2^b =g_1^a g_2^b = h$
and thus $(a+x,b+y) \in S$.

Finally we claim that $S \subseteq (a,b) + S_1$: Let $(x,y) \in S$,
i.e. $g_1^x g_2^y = h$. Since $g_1^{-a} h g_2^{-b}=1$, we get
$g_1^{x-a}g_2^{y-b}=g_1^{-a}(g_1^xg_2^y)g_2^{-b}=g_1^{-a}hg_2^{-b}=1$
and therefore $(x-a,y-b) \in S_1$. Hence $S = (a,b)+S_1$.
\end{proof}
For a knapsack-semilinear group $G$ and a finite generating set $\Sigma$ for $G$ we define a growth function. For $n \in \N$ let
$\mathsf{Knap}(n)$ (resp., $\mathsf{Exp}(n)$) be the finite set of all knapsack expressions (resp., exponent expression) $E$ over $\Sigma$ such that
$\Sol_G(E) \neq \emptyset$ and $|E| \leq n$.
We define the mapping $\mathsf{K}_{G,\Sigma}  \colon    \N \rightarrow \N$ and $\mathsf{E}_{G,\Sigma}  \colon  \N \rightarrow \N$ as follows:
\begin{eqnarray}
  \mathsf{K}_{G,\Sigma}(n) &=& \max \{ \|\Sol_G(E)\| \mid E \in \mathsf{Knap}(n) \}, \label{def-K} \\
  \mathsf{E}_{G,\Sigma}(n) &=& \max \{ \|\Sol_G(E)\| \mid E \in \mathsf{Exp}(n) \}.  \label{def-E}
\end{eqnarray}
Clearly, if $\Sol_G(E) \neq \emptyset$  and $\|\Sol_G(E)\| \le N$ then $E$ has a $G$-solution $\nu$
such that $\nu(x) \leq N$ for all variables $x$ that occur in $E$. Therefore, if $G$ has a decidable word problem and
we have a computable bound on the function $\mathsf{K}_{G,\Sigma}$ then we obtain a nondeterministic  
algorithm for  $\KP(G)$:
given a knapsack expression $E$ with variables from $X$ we can guess $\nu  \colon  X \to \N$ with $\sigma(x) \le N$ for all variables
$x$ and then verify (using an algorithm for the word problem), whether $\nu$ is indeed a solution.

Let  $\Sigma$ and $\Sigma'$ be two generating sets for the group $G$. Then there is a constant $c$ such that
$\mathsf{K}_{G,\Sigma}(n) \le \mathsf{K}_{G,\Sigma'}(cn)$, and similarly for $\mathsf{E}_{G,\Sigma}(n)$. To see this, note that 
for every $a\in \Sigma'$ there is a word $w_a\in \Sigma^*$ such that $a$ and $w_a$ represent the same element in $G$. 
Then we can choose $c=\max \{ |w_a| \mid a\in \Sigma'\}$. Due to this fact, we do not have to specify
the generating set $\Sigma$ when we say that $\mathsf{K}_{G,\Sigma}$ (resp., $\mathsf{E}_{G,\Sigma}$) is
polynomially/exponentially bounded.

We will need the following simple lemma:
\begin{lemma}
\label{lem:exp-eq}
Let $H$ be knapsack-semilinear and
let $E = v_0 (u_1^{k_1})^{x_1} v_1 (u_2^{k_2})^{x_2} v_2 \cdots (u_d^{k_d})^{x_d} v_d$
be an exponent expression over $H$ where $k_1, \dots, k_d \le k$ and $|v_0 u_1 v_1 \cdots u_d v_d| = n$.
Then the magnitude of $\Sol_H(E)$ is $(n \cdot \max\{ \mathsf{K}_H(n), k \} + 1)^{\O(n)}$.
\end{lemma}

\begin{proof}
Let $X = \{x_1, \dots, x_d\}$ (some of the variables $x_i$ might be equal)
and $Y = \{y_1, \dots, y_d\}$ be a set of $d$ {\em distinct} variables.
Then $\nu \colon X \to \N$ is a solution of $E = 1$ if and only if
$\mu \colon Y \to \N$ is a solution of
$E' = v_0 u_1^{y_1} v_1 u_2^{y_2} v_2 \cdots u_d^{y_d} v_d = 1$
where $\mu(y_i) = k_i \nu(x_i)$.
Notice that $E'$ is a knapsack expression.
Hence $\Sol_H(E)$ can be obtained as a projection of the intersection of
$\Sol_H(E')$ with a semilinear set of magnitude $\le k$ (it has to ensure that $\mu(y_i)$ is a multiple of $k_i$ and that
$\mu(y_i)/k_i = \mu(y_j)/k_j$ whenever $x_i = x_j$).
Therefore $\|\Sol_H(E)\| = (n \cdot \max\{ \mathsf{K}_H(n), k \} + 1)^{\O(n)}$.
\end{proof}
Important for us is also the following result from \cite{GanardiKLZ18}:
\begin{theorem}[\cite{GanardiKLZ18}] \label{thm:closure-wreath}
If $G$ and $H$ are knapsack-semilinear then also $G \wr H$ is knapsack-semilinear.
\end{theorem}
The proof of this result in \cite{GanardiKLZ18} does not yield a good bound of $\mathsf{K}_{G\wr H}(n)$
in terms of $\mathsf{K}_{G}(n)$ and $\mathsf{K}_{H}(n)$ (and similarly for the $\mathsf{E}$-function).
One of our main achievements will be such a bound for the special case that the left factor $G$ is f.g.~abelian.
For $\mathsf{E}_G(n)$ we then have the following bound,  which follows
from well-known bounds on solutions of linear Diophantine equations \cite{vZGS78}:
\begin{lemma}
\label{lem:abelian}
If $G$ is a f.g.~abelian group then $\mathsf{E}_G(n) \le 2^{n^{\O(1)}}$.
\end{lemma}

\subsection{Power word problem}
  
A \emph{power word} (over $\Sigma$) is a tuple $(u_1,k_1,u_2,k_2,\ldots,u_d,k_d)$ where 
$u_1, \dots, u_d \in \Sigma^*$ are words over the group generators (called the periods of the power word)
and $k_1, \dots, k_d\in \mathbb{Z}$ are integers that are given in binary notation. Such a power word represents the 
word $u_1^{k_1} u_2^{k_2}\cdots u_d^{k_d}$. Quite often, we will identify the power word $(u_1,k_1,u_2,k_2,\ldots,u_d,k_d)$
with the word $u_1^{k_1} u_2^{k_2}\cdots u_d^{k_d}$. Moreover, if $k_i=1$, then we usually omit the exponent $1$ in a power word. 
The \emph{power word problem} for the f.g.~group $G$, $\PowWP(G)$ for short, is defined as follows:
\begin{description}
\item[Input] A power word $(u_1,k_1,u_2,k_2,\ldots,u_d,k_d)$.
\item[Question] Does $u_1^{k_1} u_2^{k_2}\cdots u_d^{k_d}=1$ hold in $G$?
\end{description}
Due to the binary encoded exponents, a power word can be seen as a succinct description 
of an ordinary word. We have the following simple lemma:

\begin{lemma} \label{lemma-powWP-to-ExpEq}
If the f.g.~group $G$ is knapsack-semilinear,
$\mathsf{E}_{G}(n)$ is exponentially bounded, and $\PowWP(G)$ belongs to $\NP$ then $\ExpEq(G)$ belongs to $\NP$.
\end{lemma}

\begin{proof}
Let us consider a list $E_1, \ldots, E_k$ of exponent expressions over the group $G$
and let $n = \sum_{i=1}^k |E_i|$ be the total input length.
With Lemma~\ref{lem:cap} it follows that $\bigcap_{i=1}^k \Sol_G(E_i) \neq \emptyset$ if and only if there is some 
$\nu \in \bigcap_{i=1}^k \Sol_G(E_i)$ with $\nu(x) \leq 2^{n^{\O(1)}}$ for all variables $x$.
We can therefore guess the binary encodings of all numbers $\nu(x)$ in polynomial time and then verify
in polynomial time whether $\nu(E_i)=1$ (which is an instance of $\PowWP(G)$) for all $1 \leq i \leq k$.
\end{proof}

\section{Wreath products of nilpotent groups and $\pmb{\mathbb{Z}}$}
  
The {\em lower central series} of a group $G$ is the sequence of groups
$(G_i)_{i\geq 0}$ with $G_0 = G$ and $G_{i+1} = [G_i,G]$.
 The group $G$ is called {\em nilpotent} if there exists a $c \geq 0$ such that $G_c=1$;
 in this case the minimal number $c$ with $G_c=1$ is called the {\em nilpotency class}
 of $G$. In this section we prove Theorems~\ref{thm:nilpotent-power} and \ref{thm:nilpotent-KP} from the introduction.
 For the proofs of Theorems~\ref{thm:nilpotent-power} and \ref{thm:nilpotent-KP} we first have to consider periodic words over $G$
that were also used in \cite{GanardiKLZ18}.

\subsection{Periodic words over groups} \label{sec-periodic}

Let $G = \langle \Sigma \rangle$ be a f.g.~group.
Let $G^\omega$ be the set of all functions $f \colon \N \to G$,
which forms a group by pointwise multiplication $(fg)(t) = f(t) \cdot g(t)$.
A function $f \in G^\omega$ is \emph{periodic} if there exists a number $d \geq 1$ 
such that $f(t) = f(t+d)$ for all $t \geq 0$.
The smallest such number $d$ is called the {\em period} of $f$.
If $f \in G^\omega$ has period $d$ and $g \in G^\omega$ has period $e$
then $fg$ has period at most $\mathrm{lcm}(d,e)$.
A periodic function $f \in G^\omega$ with period $d$ can be specified by its initial $d$ elements
$f(0), \ldots, f(d-1)$ where each element $f(t)$ is given as a word over
the generating set $\Sigma$. The {\em periodic words problem} $\textsc{Periodic}(G)$ over $G$ is defined as follows:

\begin{description}
\item[Input] Periodic functions $f_1, \ldots, f_m \in G^\omega$ and a binary encoded number $T$.
\item[Question] Does the product $f = \prod_{i=1}^m f_i$ satisfy $f(t) = 1$ for all $t \le T$? 
\end{description}

The main result of this section is:

\begin{theorem} \label{thm:nilpotent-membership}
If $G$ is a f.g.~nilpotent group
then $\textsc{Periodic}(G)$ belongs to $\TC^0$.
\end{theorem}
Previously it was proven that $\textsc{Periodic}(G)$  belongs to $\TC^0$
if $G$ is abelian \cite{GanardiKLZ18}.
As an introduction let us reprove this result.

Let $\rho \colon G^\omega \to G^\omega$ be the {\em shift}-operator,
i.e. $(\rho(f))(t) = f(t+1)$, which is a group homomorphism.
For a subgroup $H$ of $G^{\omega}$, we denote by $H^{(n)}$ the smallest subgroup
of $G^{\omega}$ that contains $\rho^{0}(H), \rho^{1}(H), \ldots, \rho^n(H)$.
Note that $(H^{(m)})^{(n)}=H^{(m+n)}$ for any $m,n\in\N$.
A function $f \in G^\omega$ {\em satisfies a recurrence of order $d \geq 1$}
if $\rho^d(f)$ is contained in the subgroup $\langle f\rangle^{(d-1)}$ of $G^\omega$.
If $f$ has period $d$ then $f$ clearly satisfies a recurrence of order $d$.

Let us now consider the case that $G$ is abelian. Then, also $G^\omega$ is abelian
and we use the additive notation for $G^\omega$.
The following lemma is folklore:

\begin{lemma}[cf. \cite{Herlestam85}]
	\label{lem:abelian-recurrence}
	Let $G$ be a f.g.~abelian group.
	If $f_1, \dots, f_m \in G^\omega$ satisfy recurrences of order $d_1, \dots, d_m \geq 1$
	respectively,
	then $\sum_{i=1}^m f_i$ satisfies a recurrence of order $\sum_{i=1}^m d_i$.
\end{lemma}

\begin{proof}
Observe that $G^\omega$ is a $\Z[x]$-module with
scalar multiplication
\begin{equation} \label{eq-shift-module}
\sum_{i=0}^d a_i x^i \cdot f \mapsto \sum_{i=0}^d a_i \rho^i(f).
\end{equation}
Then $f \in G^\omega$ satisfies a recurrence of order $d \geq 1$ if and only if
there exists a monic polynomial $p \in \Z[x]$ of degree $d$ (where monic means that
the leading coefficient is one) such that $pf = 0$.
Therefore, if $p_1, \dots, p_m \in \Z[x]$ 
such that $\mathrm{deg}(p_i) = d_i \geq 1$ and $p_if_i = 0$ 
then $\prod_{i=1}^m p_i \sum_{j=1}^m f_j =
\sum_{j=1}^m (\prod_{i=1}^m p_i) f_j = 0$.
Since $\prod_{i=1}^m p_i$ is a monic polynomial of degree $d := \sum_{i=1}^m d_i$,
$\sum_{i=1}^m f_i$ satisfies a recurrence of order $d$. 
\end{proof}
The above lemma implies that $\sum_{i=1}^m f_i = 0$ if and only if  $\sum_{i=1}^m f_i(t) = 0$ for all $0 \le t \le d-1$,
where $d$ is the sum of the periods of the $f_i$.

Let us now turn to the nilpotent case.
For $n\in\N$, let $G^{\omega,n}$ be the subgroup of $G^{\omega}$
generated by all elements with period at most $n$. Then $G^{\omega,n}$
is closed under shift.
The key fact for showing \cref{thm:nilpotent-membership} is the following.
\begin{proposition}\label{prop:nilpotent-recurrence}
  If $G$ is a f.g.~nilpotent group, then there is a polynomial $p$
  such that every element of $G^{\omega,n}$ satisfies a recurrence of
  order $p(n)$.
\end{proposition}
Let $H \le G^\omega$ be a subgroup which is closed under shifting,
i.e. $\rho(H) \subseteq H$. Since the shift is a homomorphism,
the commutator subgroup $[H,H]$ is closed under shifting as well.
We will work in the abelianization $H' = H/[H,H]$ where we write $\bar f$ for the coset $f [H,H]$.
We also define $\rho \colon H' \to H'$ by $\rho(\bar f) = \overline{\rho(f)}$.
This is well-defined since $fg^{-1} \in [H,H]$
implies $\rho(f)\rho(g)^{-1} = \rho(fg^{-1}) \in [H,H]$
and hence $\overline{\rho(f)} = \overline{\rho(g)}$.
As an abelian group $H'$ is a $\Z$-module and, in fact,
$H'$ forms a $\Z[x]$-module using the shift-operator.
By the above remark (see \eqref{eq-shift-module}) we have the following
(where we use the multiplicative notation for $H'$):

\begin{lemma}\label{zx-module}
	$H'$ is a $\Z[x]$-module with the scalar multiplication
	$\sum_{i=0}^d a_i x^i \cdot \bar f \mapsto \prod_{i=0}^d \rho^i(\bar f)^{a_i}$.
\end{lemma}
Our first step for proving \cref{prop:nilpotent-recurrence} is to show that
every element of $G^{\omega,n}$ satisfies a polynomial-order recurrence,
modulo some element in $[G^{\omega,n},G^{\omega,n}]$.
\begin{lemma}\label{recurrence-up-to-commutator}
  For every $f\in G^{\omega,n}$, we have
  $\rho^d(f)\in \langle f\rangle^{(d-1)}
  [G^{\omega,n},G^{\omega,n}]$ for $d=n(n+1)/2$.
\end{lemma}
\begin{proof}
  Suppose $f=f_1\cdots f_m$ such that $f_1,\ldots,f_m\in G^{\omega}$
  are elements of period $\le n$. According to \cref{zx-module}, we
  consider $G^{\omega,n}/[G^{\omega,n},G^{\omega,n}]$ as a
  $\Z[x]$-module.
  
  If $g \in G^\omega$ has period $q$ then $\rho^q(g)g^{-1} = 1$ and
  thus $(x^q-1)\bar g = \rho^q(\bar g) \bar g^{-1} = 1$.  Define the
  polynomial $p(x) = \prod_{i=1}^n (x^i-1) = \sum_{i=0}^d a_i x^i$ of
  degree $d = n(n+1)/2$ satisfying $a_d = 1$.  Since all functions
  $f_1, \dots, f_m$ have period at most $n$ we have $p \bar{f} = 1$.
  Written explicitly we have
  \[
    1 = p \bar{f} = \prod_{i=0}^d \rho^i(\bar f)^{a_i} =
    \overline{\prod_{i=0}^d \rho^i(f)^{a_i}} \] where the order in
  the product $\prod_{i=0}^d \rho^i(f)^{a_i}$ is arbitrary.
  Noticing that $a_d=1$, we can write $\rho^d(f) = gh$ for some
  $g\in\langle f\rangle^{(d-1)}$ and
  $h\in [G^{\omega,n},G^{\omega,n}]$, which has the desired form.
\end{proof}
The following \lcnamecref{lem:HH-generators} gives us control over the remaining factor from $[G^{\omega,n},G^{\omega,n}]$.
\begin{lemma}
	\label{lem:HH-generators}
	Let $G$ be a group with nilpotency class $c$. Then
        $[G^{\omega,n},G^{\omega,n}]\subseteq [G,G]^{\omega,n^{2c}}$.
\end{lemma}
\begin{proof}
We need the fact that the commutator subgroup $[F,F]$ of a group $F$ with generating set $\Gamma$ is generated by all left-normed commutators
\[
	[g_1, \dots, g_k] := [[ \dots [[g_1,g_2],g_3], \dots], g_k ]
\]
where $g_1, \dots, g_k \in \Gamma \cup \Gamma^{-1}$ and $k \ge 2$,
cf. \cite[Lemma~2.6]{clement2017theory}.
Therefore $[G^{\omega,n},G^{\omega,n}]$ is generated by all left-normed commutators
$[g_1, \dots, g_k]$ where $k \ge 2$ and $g_1, \dots, g_k \in G^\omega$ have period at most $n$.
Furthermore, we can bound $k$ by $c$ since any left-normed commutator
$[g_1, \dots, g_{c+1}]$ is trivial (recall that $G$ is nilpotent of class $c$).

A left-normed commutator $[g_1, \dots, g_k]$ with $2 \leq k \leq c$ and 
$g_1, \dots, g_k$ periodic with period at most $n$ is 
a product containing at most $2k \le 2c$ distinct functions of period at most $n$
(namely, the $g_1, \dots, g_k$ and their inverses).
Hence $[G^{\omega,n},G^{\omega,n}]$ is generated by functions $g \in [G,G]^\omega$ of period at most $n^{2c}$.
\end{proof}
We are now ready to prove \cref{prop:nilpotent-recurrence}.
\begin{proof}
  The \lcnamecref{prop:nilpotent-recurrence} is proved by induction on
  the nilpotency class of $G$.  If $G$ has nilpotency class 0 then $G$
  is trivial and the claim is vacuous.  Now suppose that $G$ has
  nilpotency class $c \ge 1$. According to
  \cref{recurrence-up-to-commutator}, we have
  $\rho^d(f)\in\langle f\rangle^{(d-1)} h$ for some
  $h\in [G^{\omega,n},G^{\omega,n}]$. By \cref{lem:HH-generators}, we
  have $[G^{\omega,n},G^{\omega,n}]\subseteq
  [G,G]^{\omega,n^{2c}}$. Since the group $[G,G]$ has nilpotency class at most
  $c-1$,\footnote{We could not find a proof for this fact in the literature, so let us provide 
  the argument: Define $G_0 = G$ and $G_{i+1} = [G_i,G]$ and 
  $H_0 = [G,G]$ and $H_{i+1} = [H_i,[G,G]]$. It suffices to show that $H_i \leq G_{i+1}$ for all $i \geq 0$.
  For $i=0$ this is follows from the definition. For the induction step let $i>0$. We get
  $H_i =  [H_{i-1},[G,G]] \leq [G_i,G] = G_{i+1}$.}
   we may apply induction. Thus, we know that
  $\rho^e(h)\in \langle h\rangle^{(e-1)}$ for some $e=e(n^{2c})$.
  We claim that then $\rho^{d+e}(f)\in\langle f\rangle^{(d+e-1)}$.
  Note that
  \[ \rho^{d+e}(f)\in \rho^{e}(\langle f\rangle^{(d-1)} h)\subseteq \rho^e(\langle f\rangle^{(d-1)})\rho^e(h)\subseteq\langle f\rangle^{(d+e-1)}\cdot\rho^e(h). \]
  Therefore, it suffices to show that
  $\rho^e(h)\in \langle f\rangle^{(d+e-1)}$. Since
  $\rho^d(f)\in\langle f\rangle^{(d-1)}h$ we have $h\in \langle f\rangle^{(d)}$
  and thus   $\rho^e(h)\in \langle h\rangle^{(e-1)}\subseteq (\langle f\rangle^{(d)})^{(e-1)}=\langle f\rangle^{(d+e-1)}$.
\end{proof}

\begin{proof}[Proof of \cref{thm:nilpotent-membership}]
	Given periodic functions $f_1, \dots, f_m \in G^\omega$
	with maximum period $n$, and a number $T \in \N$.
	By \Cref{prop:nilpotent-recurrence}
	the product $f = f_1 \cdots f_m$ satisfies a recurrence of order $d$,
	where $d$ is bounded polynomially in $n$.
	Notice that $f = 1$ if and only if $f(t) = 1$ for all $t \le d-1$.
	Hence, it suffices to verify that $f_1(t) \cdots f_m(t) = 1$ for all $t \le \min\{d,T\}$.
	This can be accomplished by solving in parallel a polynomial number
	of instances of the word problem over $G$, which is contained in $\TC^0$ by \cite{MyasnikovW17}.
\end{proof}

 \subsection{Proofs of Theorems~\ref{thm:nilpotent-power} and \ref{thm:nilpotent-KP}}
 
 Let us start with the proof of Theorem~\ref{thm:nilpotent-power}.
The following result is from \cite{LoWe19}.
 
\begin{proposition}[\cite{LoWe19}] \label{prop:TC0-reduction}
For every f.g.~group $G$,  the problem $\PowWP(G \wr \mathbb{Z})$ belongs to 
$\TC^0(\textsc{Periodic}(G),\PowWP(G))$.
\end{proposition}
The following proposition is from \cite{GanardiKLZ18}  (see the proof of Proposition~7.2 in \cite{GanardiKLZ18}).

 \begin{proposition}[\cite{GanardiKLZ18}] \label{prop:NP-reduction}
Let $G$ be a f.g.~group. There is a non-deterministic polynomial time Turing machine $M$ that takes as input
a knapsack expression $E$ over $G \wr \mathbb{Z}$ and outputs in each leaf of the computation tree the following data:
\begin{itemize}
\item an instance of $\ExpEq(G)$ and
\item a finite list of instances of $\textsc{Periodic}(G)$.
\end{itemize}
Moreover, the input expression $E$ has a $(G \wr \mathbb{Z})$-solution if and only if there is a leaf in the computation tree of $M$
such that all instances that $M$ outputs in this leaf are positive.
\end{proposition}

\begin{proof}[Proof of Theorem~\ref{thm:nilpotent-power}]
By \cite{LoWe19} the power word problem for a f.g. nilpotent group belongs to $\TC^0$ and by
Theorem~\ref{thm:nilpotent-membership}, $\textsc{Periodic}(G)$ belongs to $\TC^0$.
The theorem follows from Proposition~\ref{prop:TC0-reduction}.
\end{proof}

\begin{proof}[Proof of Theorem~\ref{thm:nilpotent-KP}]
Let $G$ be a finite nontrivial nilpotent group. By Theorem~\ref{thm:NP-hard-KP}, knapsack for
$G \wr \mathbb{Z}$ is $\NP$-hard. Moreover, $\textsc{Periodic}(G)$ belongs to $\TC^0$ and
$\ExpEq(G)$ belongs to $\NP$ (this holds for every finite group). 
Proposition~\ref{prop:NP-reduction} implies that $\KP(G \wr \mathbb{Z})$ belongs to $\NP$.
\end{proof}

 \section{Wreath products with abelian left factors}

In this section we prove Theorems~\ref{thm:iterated-pwp} and \ref{thm:iterated-kp}.
For this, we prove two transfer results.
For a finitely generated group $G = \langle \Sigma \rangle$
we define the \emph{power compressed power problem} $\PowPP(G)$
as the following computational problem.
\begin{description}
\item[Input] A word $u \in \Sigma^*$ and a power word $(v_1,k_1, \ldots, v_d,k_d)$ over $\Sigma$.
\item[Output] A binary encoded number $z \in \Z$ with $u^z = v$
where $v = v_1^{k_1} \dots v_d^{k_d}$, or \textbf{no} if $u^z = v$ has no solution. 
\end{description}
Notice that if $G$ is torsion-free then $u^x = v$ has at most one solution whenever $u \neq 1$.

We say that a group $G = \langle \Sigma \rangle$ is {\em tame with respect to commensurability},
or short {\em c-tame}, if there exists a number $d \in \N$ such that
for all commensurable elements $g,h \in G$ having infinite order
there exist numbers $s,t \in \Z \setminus \{0\}$ such that
$g^s = h^t$ and $|s|,|t| \le \O((|g|+|h|)^d)$.

\begin{theorem}\label{thm:ppp-wr}
Let $H$ and $A$ be f.g.~groups where $A$ is abelian
and $H$ is c-tame and torsion-free.
Then $\PowPP(A \wr H)$ is $\TC^0$-reducible to $\PowPP(H)$.
\end{theorem}
Later, we will show how to derive Theorem~\ref{thm:iterated-pwp} from 
Theorem~\ref{thm:ppp-wr}. For Theorem~\ref{thm:iterated-kp} we need the 
following transfer theorem (recall the definition of an orderable group from Section~\ref{sec-groups} and the definition
of the function $\mathsf{E}_G(n)$ from \eqref{def-E} in Section~\ref{sec-knapsack-semilinear}):

\begin{theorem}
	\label{thm:abelian-wr}
	Let $H$ and $A$ be f.g.~groups
	where $A$ is abelian and $H$ is orderable and knapsack-semilinear.
	If $\mathsf{E}_H(n)$ is exponentially bounded then so is $\mathsf{E}_{A \wr H}(n)$.
\end{theorem}
Using Theorem~\ref{thm:iterated-pwp}  and \ref{thm:abelian-wr} we can prove 
Theorem~\ref{thm:iterated-kp}: let us fix an iterated wreath product $W = W_{m,r}$ for some $m \geq 0$, $r \geq 1$
(recall that $W_{m,r} = \Z^r$ and $W_{m+1,r} = \Z^r \wr W_{m,r}$).
Since $\Z^m$ is orderable, Theorem~\ref{thm-wreath-orderable} implies that $W$ is orderable. 
Moreover, by Theorem~\ref{thm:closure-wreath}, $W$ is also knapsack-semilinear.
Since by Lemma~\ref{lem:abelian}, $\mathsf{E}_{A}(n)$ is exponentially bounded for every
f.g.~abelian group $A$, it follows from Theorem~\ref{thm:abelian-wr} that $\mathsf{E}_{W}(n)$
is exponentially bounded as well. By Theorem~\ref{thm:iterated-pwp} and Lemma~\ref{lemma-powWP-to-ExpEq},
$\ExpEq(W)$ belongs to $\NP$. Finally, $\NP$-hardness of $\ExpEq(W)$ follows from the fact that the question whether 
a given system of linear Diophantine equations with unary encoded numbers has a solution in $\N$
is $\NP$-hard.

Before we start the proofs of Theorems~\ref{thm:ppp-wr} and \ref{thm:abelian-wr} we show some simple
normalization results and introduce the concept of a progression in a torsion-free group.

\subparagraph{Normalization.}

Consider a wreath product  $G=A \wr H$, where $A$ is abelian. 
We will show how to bring an exponent expression (resp., a power word) 
into a particular form that will be useful later.

An exponent expression $E = v_0 u_1^{x_1} v_1 u_2^{x_2} v_2 \cdots u_d^{x_d} v_d$ over $G$
is {\em normalized} if
\begin{enumerate}[(i)]
\item $u_i \in AH$ for all $1 \le i \le d$ (here $AH$ is $\{ ah \mid a \in A, h \in H\}$)
\item $v_i \in H$ for all $0 \le i \le d$, and
\item $v_0 = 1$.
\end{enumerate}
By the following lemma we can assume normalized exponent expressions
in order to prove Theorem~\ref{thm:abelian-wr}.

\begin{lemma} \label{lemma-normalize2}
Let $E$ be an exponent expression over $G=A \wr H$ of length $n$ and assume that $H$ is knapsack-semilinear.
There exists a normalized exponent expression $E''$ such that
$\|\Sol_G(E)\| = (2 (n+1)  \|\Sol_G(E'') \|+1)^{\O(n)}$ and $|E''| \leq \O(n^2)$.
\end{lemma}

\begin{proof}
Note that by Theorem~\ref{thm:closure-wreath} also $G$ is knapsack-semilinear.
Property (iii) can always be established by conjugating with $v_0$.
Hence we can focus on properties (i) and (ii).

We first explain how to achieve property (i) for a given power $u^x$.
Since $u$ is given as a word over the generators of $A$ and $H$ we can
factorize $u$ as $u = g_0 g_1 \cdots g_\ell$ where $g_0 \in H$ and $g_1, \dots, g_\ell \in AH$.
Let us write $\sigma_{i,j} = \sigma(g_i \cdots g_j)$ for $i \leq j$
and $\sigma_{i,j} = 1$ for $i > j$.
Then, for every $x$ we have
\begin{align}
	u^x & = \left(\prod_{i=1}^\ell (\sigma_{0,i-1} \, g_i \, \sigma_{i+1,\ell})^x \sigma(u)^{-x}\right) \sigma(u)^x \nonumber \\ \label{eq:u-tilde}
	& = \left(\prod_{i=1}^\ell \sigma_{0,i-1} (g_i \, \sigma_{i+1,\ell} \sigma_{0,i-1})^x \sigma_{0,i-1}^{-1}  \sigma(u)^{-x}\right) \sigma(u)^x
\end{align}
Notice that $g_i \, \sigma_{i+1,\ell} \sigma_{0,i-1}\in AH$
and $\sigma(u) \in H \subseteq AH$.
Let $\tilde u$ be the expression from \eqref{eq:u-tilde},
which has length $\O(|u|^2)$.

Now let $E = v_0 u_1^{x_1} v_1 u_2^{x_2} v_2 \cdots u_d^{x_d} v_d$ be an exponent expression
over $G$.
We construct the exponent expression
\begin{equation}
\label{exp-E'}
E' = v_0  \tilde u_1  v_1  \tilde u_2  \cdots  \tilde u_d  v_d.
\end{equation}
We have $\Sol_G(E) = \Sol_G(E')$.
Notice that the $E$ and $E'$ use the same variables and that the length of $E'$ is bounded 
by $\O(n^2)$. 

For condition (ii) from the lemma observe that every element $v \in G$  in \eqref{exp-E'} that occurs
between two consecutive powers or before (after) the first (last) power 
is given as a word over the generators of $A$ and $H$, say $v = a_1 h_1 a_2 h_2 \cdots a_k h_k$
where $a_j \in A$, $h_j \in H$, $1 \le j \le k$.
We replace $v$ by the expression
$\hat v = a_1^{y} h_1 a_2^{y} h_2 \cdots a_k^{y} h_k$ for a fresh variable $y$
and enforce $y = 1$ by a semilinear constraint.
Applying this to every such word $v$ in \eqref{exp-E'} yields an exponent expression $E''$
with at most $n+1$ variables and length $\O(n^2)$. 

We have $\Sol_G(E) = \Sol_G(E') = \pi(\Sol_G(E'') \cap C)$ where $C$ is the semilinear constraint saying that $y=1$,
and $\pi$ is the projection to the original variables of $E$.
By Lemma~\ref{lem:cap} we have 
\[ \|\Sol_G(E)\| =  \|\Sol_G(E'') \cap C\| \leq (2 (n+1)  \|\Sol_G(E'') \|+1)^{\O(n)} .
\]
This concludes the proof.
\end{proof}
A power word $(u_1, k_1, \ldots, u_d,k_d)$ over $G$
is {\em normalized} if $u_i \in AH$ for all $1 \le i \le d$.

\begin{lemma} \label{lemma-normalize-power-word}
From a given power word over $G$ one can compute in $\TC^0$ a normalized power word
that evaluates to the same group element of $G$.
\end{lemma}

\begin{proof}
We apply the same construction as in the proof of Lemma~\ref{lemma-normalize2} (where of course the variables
in the exponents are replaced by the numbers from the power word). The new variable $y$ in the above proof is of 
course replaced by the exponent $1$. Finally, notice that the 
transformation from $E$ to $E''$ can be carried out in $\TC^0$. 
\end{proof}

\subparagraph{Progressions.}
A {\em progression} over a torsion-free group $H$ is a non-empty finite sequence $\bm{p} = (p_i)_{0 \le i \le k}$
of the form $p_i = ab^i$ where $a,b \in H$.
We call $a$ the {\em offset},
$b$ the {\em period}\footnote{A progression of length two or more has a unique period; progressions of length one are assigned
a fixed but arbitrary period.}
and define $\supp(\bm{p}) = \{ p_i \mid 0 \le i \le k \}$.
The {\em length} of $\bm{p}$ is $|\bm{p}|=k+1$ and its {\em endpoints} are $p_0$ and $p_k$.
A progression whose period is nontrivial is called a {\em ray}. 
Since $H$ is torsion-free, all entries of a ray are pairwise distinct.
Two rays are {\em parallel} if their periods are commensurable (see Section~\ref{sec-commens}).

\begin{lemma}
	\label{lem:rays-cap}
	If two rays $\bm{p}$ and $\bm{q}$ are not parallel
	then $|\supp(\bm{p}) \cap \supp(\bm{q})| \le 1$.
\end{lemma}

\begin{proof}
	Let $p_i = ab^i$ and $q_j = gh^j$.
	Suppose that $|\supp(\bm{p}) \cap \supp(\bm{q})| \ge 2$.
	Then there are numbers $i \neq i'$ and $j \neq j'$ such that
	$ab^i = gh^j$ and $ab^{i'} = gh^{j'}$.
	This implies $b^{i-i'} = h^{j'-j}$, which means that $\bm{p}$ and $\bm{q}$ are parallel -- a contradiction.
\end{proof}

\subsection{Proofs of Theorem~\ref{thm:ppp-wr}}

In this section we prove Theorem~\ref{thm:ppp-wr}.
In Section~\ref{section-21->3} we then deduce 
Theorem~\ref{thm:iterated-pwp} from Theorem~\ref{thm:ppp-wr}.

\subsubsection{Reducing $\PowWP(A \wr H)$ to $\PowPP(H)$}

For the rest of this section we fix a finitely generated abelian group $A = \langle \Gamma\rangle$
and a finitely generated torsion-free group $H= \langle \Sigma\rangle$.

A {\em power-compressed ray} over $H$
is a triple $(u,v,\ell)$ where $u$ is a power word over $\Sigma$, $v \in \Sigma^*$ is a word
with $v \neq 1$ in $H$
and $\ell \in \N$ is a binary encoded number. Such a power-compressed ray $(u,v,\ell)$
defines the ray $(uv^i)_{0 \le i \le \ell}$.
We will identify the triple with the ray itself.
Define the intersection set $\Int(\bm{p},\bm{q})$ of two rays
$\bm{p},\bm{q}$ by
\[
	\Int(\bm{p},\bm{q}) = \{ i \in [0,|\bm{p}|-1] \mid \exists j \in [0,|\bm{q}|-1] \colon p_i = q_j \}.
\]
If $\bm{p},\bm{q}$ are parallel rays and $H$ is c-tame
then one can reduce the computation of $\Int(\bm{p},\bm{q})$
to $\PowPP(H)$.

\begin{lemma}
\label{lem:intersection}
If $H$ is c-tame and torsion-free then the following problem
is $\TC^0$-reducible to $\PowPP(H)$:
given two parallel power-compressed rays $\bm{p},\bm{q}$ over $H$,
decide whether $\Int(\bm{p},\bm{q})$ is non-empty and, if so,
compute an arithmetic progression $\bm{s}$ such that
$\Int(\bm{p},\bm{q}) = \supp(\bm{s})$.
\end{lemma}

\begin{proof}
Suppose that $\bm{p} = (ab^i)_{0 \le i \le k}$
and $\bm{q} = (gh^j)_{0 \le j \le \ell}$.
By c-tameness
there exists $s,t \in \Z \setminus \{0\}$
such that $\{ (i,j) \in \Z^2 \mid b^i = h^j \} = \langle (s,t) \rangle$
and $|s|,|t|$ are polynomially bounded in $|b|+|h|$.
We compute the unary encodings of such numbers $s,t$
by checking all identities $b^s = h^t$
for $|s|,|t| \leq (|b|+|h|)^{\O(1)}$
(the word problem of $H$ is a special case of $\PowPP(H)$).

Since $b\neq 1\neq h$ and $H$ is torsion-free
we must have $s \neq 0 \neq t$.
We can enforce $t > 0$ by inverting the generator $(s,t)$ if necessary.
Since $ab^i = gh^j$ is equivalent to $b^i h^{-j} = a^{-1}g$, Lemma~\ref{lemma-2dim} implies that
$\{ (i,j) \in \Z^2 \mid ab^i = gh^j \}$
is either empty or a coset of $\langle (s,t) \rangle$.
Therefore, if $ab^i = gh^j$ has any solution, then it has a solution $(i,j) \in \Z^2$
where $0 \le j \le t - 1$.
For all $0 \le t_0 \le t - 1$ we solve the $\PowPP(H)$-instance $ab^x = gh^{t_0}$.
If there is no solution for any $0 \le t_0 \le t - 1$ we can conclude $\Int(\bm{p},\bm{q}) = \emptyset$.
Otherwise let $0 \le t_0 \le t - 1$ and $s_0 \in \Z$ with $ab^{s_0} = gh^{t_0}$. We obtain the integer
$s_0$ in binary encoding.
Then $\Int(\bm{p},\bm{q})$ is the projection to the first component of the set
\[
	\big((s_0,t_0) + \langle (s,t) \rangle\big) \cap ([0,k] \times [0,\ell]).
\]
Next we compute the interval
$Y = \{ y \in \Z  \mid 0 \le t_0 + y t \le \ell \}$.
Since $t > 0$ we have
\begin{align*}
	y \in Y \iff 0 \le t_0 + y t \le \ell 
	\iff -\frac{t_0}{t} \le y \le \frac{\ell-t_0}{t}.
\end{align*}
Hence the endpoints of $Y$ are $y_1 = \lceil -t_0/t \rceil$
and $y_2 = \lfloor (\ell-t_0)/t \rfloor$,
which can be computed in $\TC^0$
since integer division is in $\TC^0$ (here, we only need
the special case, where we divide by a unary encoded integer).
If $y_1 > y_2$ then $Y$ is empty and also $\Int(\bm{p},\bm{q})$ is empty.
Otherwise, we compute $\Int(\bm{p},\bm{q})$ using the fact that
\[
	\Int(\bm{p},\bm{q}) = \{ s_0 + s y \mid y \in Y \} \cap [0,k].
\]
We transform $Y = [y_1,y_2]$ into the arithmetic progression $(s_0 + s y)_{y_1 \le y \le y_2}$
and intersect it with the interval $[0,k]$.
\end{proof}

\begin{lemma}
\label{lem:eval}
Let $H$ be torsion-free.
Then the following problem is $\TC^0$-reducible to $\PowPP(H)$:
given a power word $u$ over $\Gamma \cup \Sigma$ representing an element in $A \wr H$,
and a power word $v$ over $\Sigma$ representing an element in $H$,
compute a power word for $\tau(u)(v)$.
\end{lemma}

\begin{proof}
Let  $u = u_1^{k_1} \cdots u_d^{k_d}$.
By Lemma~\ref{lemma-normalize-power-word} we normalize $u$ in $\TC^0$ so that
for every $i$, $u_i = a_i \, \sigma(u_i)$ for some $a_i \in A$.
By Lemma~\ref{lem:tau} we have
\[
	\tau(u)(v) = \sum_{i=1}^d \tau(\sigma(u_1^{k_1} \dots u_{i-1}^{k_{i-1}}) \, u_i^{k_i})(v)
	= \sum_{i=1}^d \tau(u_i^{k_i})(\sigma(u_1^{k_1} \dots u_{i-1}^{k_{i-1}})^{-1} v).
\]
Hence it suffices to compute $\tau(u_i^{k_i})(v_i)$
for the power word $v_i =  \sigma(u_{i-1})^{-k_{i-1}} \dots \sigma(u_1)^{-k_1} v$.
If $\sigma(u_i) = 1$ then $\tau(u_i^{k_i})(v_i) = k_i \cdot a_i$ if $v_i = 1$
and $\tau(u_i^{k_i})(v_i) = 0$ otherwise.
If $\sigma(u_i) \neq 1$ we compute a solution $x \in \Z$ for
$\sigma(u_i)^x = v_i$ (this is an instance of $\PowPP(H)$).
If there is no solution or $x < 0$ or $x \ge k_i$ then $\tau(u_i^{k_i})(v_i) = 0$;
otherwise $\tau(u_i^{k_i})(v_i) = a_i$.
\end{proof}
If $\bm{s}$ is an arithmetic progression and $a \in A$
then we define $f_{\bm{s},a} \colon \N \to A$ by
$f_{\bm{s},a}(t) = a$ if $t \in \supp(\bm{s})$
and $f_{\bm{s},a}(t) = 0$ otherwise.

\begin{lemma}
\label{lem:rays}
Given a unary encoded number $b \in \N$ and
finite multiset $M$ of pairs $(\bm{s},a)$,
where $\bm{s}$ is an arithmetic progression and $a \in \Gamma^*$ is an element of $A$,
we can compute in $\TC^0$ the following:
\begin{itemize}
\item the set $T = \{ t \in \N \mid \sum_{(\bm{s},a) \in M} f_{\bm{s},a}(t) \neq 0 \}$ in case $|T| < b$,
\item $\bot$ in case $|T|\geq b$.
\end{itemize}
\end{lemma}

\begin{proof}
Let  $\bm{s} = (d+ie)_{0 \leq i \leq \ell}$ be an arithmetic progression with $d \in \N$ and 
$e \in \N \setminus \{0\}$. We define the interval $I(\bm{s}) = [d,d+\ell e]$ 
and for $a \in A$ we define the mapping $g_{\bm{s},a} \colon \N \to A$ by
\[
	g_{\bm{s},a}(t) = \begin{cases}
	a, &\text{if } t \equiv d \!\!\pmod{e}, \\
	0, &\text{otherwise}.
	\end{cases}
\]
It is easy to verify that $f_{\bm{s},a}(t) = g_{\bm{s},a}(t)$ for all $t \in I(\bm{s})$.
Let $n \in \N$ be the maximal number in any interval $I(\bm{s})$ for $(\bm{s},a) \in M$.
By Lemma~\ref{lem:interval-types} we can compute in $\TC^0$ a partition $\mathcal{J}$
of $[0,n]$ into intervals such that for all $J \in \mathcal{J}$ and all $(\bm{s},a) \in M$
we have either $J \subseteq I(\bm{s})$ or $J \cap I(\bm{s}) = \emptyset$.
In the following we will show how to either compute $T \cap J$ 
or establish that $|T \cap J| \ge b$ for all $J \in \mathcal{J}$.
Then the statement follows because
if $|T \cap J| \ge b$ for some $J \in \mathcal{J}$ then $|T| \ge b$;
otherwise we can compute $T = \bigcup_{J \in \mathcal{J}} (T \cap J)$
and return $T$ if $|T| < b$.

Let $J \in \mathcal{J}$ be arbitrary.
Then for all $t \in J$ we have
\begin{equation}
\label{eq:interval-sum}
\begin{aligned}
	\sum_{(\bm{s},a) \in M} f_{\bm{s},a}(t) 
	= \sum_{\substack{(\bm{s},a) \in M \\ J \subseteq I(\bm{s})}} f_{\bm{s},a}(t) +
	\sum_{\substack{(\bm{s},a) \in M \\ J \cap I(\bm{s}) = \emptyset}} \underbrace{f_{\bm{s},a}(t)}_{=0}
	= \sum_{(\bm{s},a) \in M_J} g_{\bm{s},a}(t) ,
\end{aligned}
\end{equation}
where $M_J = \{ (\bm{s},a) \in M \mid J \subseteq I(\bm{s}) \}$.
Recall that $g_{\bm{s},a}(t) = a$ if $t$ is congruent to the offset of $\bm{s}$
modulo its period, and otherwise $g_{\bm{s},a}(t) = 0$.
Hence, for a given input $t \in \N$ we can compute the value \eqref{eq:interval-sum}
in $\TC^0$ (input as well as output are binary encoded).
Let $p$ be the sum of all periods of all arithmetic progressions $\bm{s}$
occurring in $M$, which is linear in the input size.
If $|J| < bp$ then we can compute $T \cap J$ in $\TC^0$.
If $|J| \ge bp$ and $j = \min J$ we compute
\begin{equation}
	\label{eq:tj}
	T_J = \{ t \in [j,j+bp-1] \mid \sum_{(\bm{s},a) \in M} f_{\bm{s},a}(t) \neq 0 \},
\end{equation}
which is a subset of $T$.
If $|T_J| \ge b$, we have established $|T|\ge b$ and we can output $\bot$.
If $|T_J| < b$ then $[j,j+bp-1] \setminus T_J$
contains at least $bp - (b-1) = b(p-1) + 1$ many elements
and $[j,j+bp-1] \setminus T_J$ is a disjoint union of at most $b$ intervals.
Hence there exists an interval $I \subseteq [j,j+bp-1] \setminus T_J$ 
containing at least $p$ elements.
This implies that
\[
	0 = \sum_{(\bm{s},a) \in M} f_{\bm{s},a}(t) = \sum_{(\bm{s},a) \in M_J} g_{\bm{s},a}(t)
\]
for all $t \in I$.
Since $\sum_{(\bm{s},a) \in M_J} g_{\bm{s},a}$ satisfies a recurrence of order at most $p$
by Lemma~\ref{lem:abelian-recurrence} we know that in fact
$\sum_{(\bm{s},a) \in M_J} g_{\bm{s},a} = 0$.
By \eqref{eq:interval-sum} we have $\sum_{(\bm{s},a) \in M} f_{\bm{s},a}(t) = 0$ for all $t \in J$,
and thus we can output $T \cap J = \emptyset$.
This concludes the proof.
\end{proof}
We now come to the main reduction of this subsection:

\begin{proposition}
If the group $H$ is c-tame and torsion-free then $\PowWP(A \wr H) \in \TC^0(\PowPP(H))$.
\end{proposition}

\begin{proof}
Take a power word $u = u_1^{k_1} \cdots u_d^{k_d}$ over $(\Gamma \cup \Sigma)^*$.
By Lemma~\ref{lemma-normalize-power-word} we normalize $u$ in $\TC^0$ so that
for every $i$, $u_i = a_i \, \sigma(u_i)$ for some $a_i \in \Gamma^*$.
To test $u = 1$ we need to verify both $\sigma(u) = 1$ and $\tau(u) = 0$.
The former equation is an instance of $\PowWP(H)$.

For $1 \le i \le d$ we define
$\hat u_i = \sigma(u_1^{k_1} \cdots u_{i-1}^{k_{i-1}}) \, u_i^{k_i}$.
By Lemma~\ref{lem:tau} we know that $\tau(u) = \sum_{i = 1}^d \tau(\hat u_{i})$.
For $1 \le i \le d$ and $0 \le k < k_i$ we define
\[
	\sigma(i,k) = \sigma(u_1^{k_1} \cdots u_{i-1}^{k_{i-1}} u_i^k).
\]
Notice that $\bm{p}_i = (\sigma(i,k))_{0 \le k < k_i}$ is a power-compressed progression,
and, since $u$ is normalized, we have $\supp(\tau(\hat u_i)) \subseteq \supp(\bm{p}_i)$ (we have equality
if $a_i \neq 0$).
Hence it suffices to test whether $\tau(u)(h) = 0$ for all $h \in \supp(\bm{p}_i)$ and $i \in [1,d]$.

Let $R = \{ i \in [1,d] \mid \sigma(u_i) \neq 1 \}$
and define the equivalence relation $\parallel$ on $R$
by $i \parallel j$ if and only if $\sigma(u_i)$ and $\sigma(u_j)$
are commensurable, or equivalently if the rays $\bm{p}_i$
and $\bm{p}_j$ are parallel.
For all $i,j \in R$ with $i \parallel j$ we compute
$\Int(\bm{p}_i,\bm{p}_j)$ as an arithmetic progression $\bm{s}_{i,j}$.
By Lemma~\ref{lem:intersection} this can be accomplished by a $\TC^0$-reduction to $\PowPP(H)$.
If $t \in \Int(\bm{p}_i,\bm{p}_j)$ then $\sigma(i,t) \in \supp(\bm{p}_j)$
and therefore $f_{\bm{s}_{i,j},a_j}(t) = a_j = \tau(\hat u_j)(\sigma(i,t))$.
If $t \in [0,k_i-1] \setminus \Int(\bm{p}_i,\bm{p}_j)$
then $\sigma(i,t) \notin \supp(\bm{p}_j)$ and therefore
$f_{\bm{s}_{i,j},a_j}(t) = 0 = \tau(\hat u_j)(\sigma(i,t))$.
Hence we have shown that
\begin{equation}
 f_{\bm{s}_{i,j},a_j}(t) = \tau(\hat u_j)(\sigma(i,t)), \quad \text{for all } 0 \le t \le k_i - 1.
\end{equation}
For all $i \in R$ we define
\begin{equation}
	\label{eq:ts}
	T_i = \{ t \in [0,k_i-1] \mid \sum_{i \parallel j} \tau(\hat u_j)(\sigma(i,t)) \neq 0 \}.
\end{equation}
By Lemma~\ref{lem:rays} for all $i \in R$ we can either compute the set $T_i$
or conclude that $|T_i| \ge d+1$.

If there exists $i \in R$ with $|T_i| \ge d+1$ then we claim that $\tau(u) \neq 0$:
We say that an index $j \in [1,d]$ {\em crosses} $t \in T_i$ if
$i \nparallel j$ and $\sigma(i,t) \in \supp(\bm{p}_j)$ (note that if $j \notin R$ then $i \nparallel j$ holds).
Notice that a single index $j \in [1,d]$ can cross at most one element $t \in T_i$
since otherwise $|\supp(\bm{p}_i) \cap \supp(\bm{p}_j)| \ge 2$,
which contradicts Lemma~\ref{lem:rays-cap}.
This implies that $T_i$ contains at most $d$ crossed elements
and therefore at least one uncrossed element, say $t \in T_i$.
Since $\sigma(i,t) \notin \supp(\bm{p}_j) \supseteq \supp(\tau(\hat u_j))$
for all $j \in [1,d]$ with $i \nparallel j$ we obtain
\[
	\tau(u)(\sigma(i,t)) = \sum_{j = 1}^d \tau(\hat u_j)(\sigma(i,t)) =
	\underbrace{\sum_{i \parallel j} \tau(\hat u_j)(\sigma(i,t))}_{\text{$\neq 0$ by \eqref{eq:ts} }} + 
	\underbrace{\sum_{i \nparallel j} \tau(\hat u_j)(\sigma(i,t))}_{\text{$=0$}} \neq 0,
\]
which shows the claim.

In the other case we have computed all sets $T_i$ for $i \in R$.
Using Lemma~\ref{lem:eval} we test in $\TC^0$ whether
\begin{equation}
	\label{eq:x1}
	\tau(u)(\sigma(i,t)) = 0, \quad \text{for all } t \in T_i \text{ and } i \in R
\end{equation}
and whether
\begin{equation}
	\label{eq:x2}
	\tau(u)(\sigma(i,0)) = 0, \quad \text{for all } i \in [1,d] \setminus R
\end{equation}
holds.
If any of the equalities in \eqref{eq:x1} and \eqref{eq:x2} does not hold we know that $\tau(u) \neq 0$.
Otherwise we can verify that $\tau(u) = 0$:
Let $h \in \bigcup_{1 \le i \le d} \supp(\bm{p}_i)$.
If $h = \sigma(i,0)$ for some $i \in [1,d] \setminus R$
or $h = \sigma(i,t)$ for some $t \in T_i$ and $i \in R$
we are done by \eqref{eq:x1} and \eqref{eq:x2}.
Now assume the contrary.
Then we know $\tau(\hat u_j)(h) = 0$ for all $j \in [1,d] \setminus R$.
We have
\begin{equation}
	\label{eq:tau-sum}
	\tau(u)(h) = \sum_{j \in R} \tau(\hat u_j)(h) + 
	\sum_{j \in [1,d] \setminus R} \tau(\hat u_j)(h) = \sum_{C \text{ a } \parallel\text{-class}} \ \sum_{j \in C} \tau(\hat u_j)(h),
\end{equation}
and we claim that $\sum_{j \in C} \tau(\hat u_j)(h) = 0$
for all $\parallel$-classes $C$.
Consider a $\parallel$-class $C$.
If $h = \sigma(i,t)$ for some $i \in C$ and $t \in [0,k_i-1]$
then
\[
	\sum_{j \in C} \tau(\hat u_j)(h) = \sum_{i \parallel j} \tau(\hat u_j)(\sigma(i,t)) = 0,
\]
since $t \notin T_i$.
Otherwise $h \notin \{ \sigma(i,t) \mid i \in C, \, t \in [0,k_i-1] \} = \bigcup_{i \in C} \supp(\bm{p}_i)$,
and therefore $\tau(\hat u_j)(h) = 0$ for all $j \in C$.
By \eqref{eq:tau-sum} we conclude that $\tau(u) = 0$.
\end{proof}

\subsubsection{Reducing $\PowPP(A \wr H)$ to $\PowWP(A \wr H)$ and $\PowPP(H)$}

\begin{lemma} \label{lemma:cyclic-red-pwp}
If the finitely generated group $H$ is torsion-free then $\PowPP(A \wr H)$ belongs to $\TC^0(\PowWP(A \wr H), \PowPP(H))$.
\end{lemma}

\begin{proof}
We want to solve $u^x=v$ in $A \wr H$, where $u \in (\Gamma \cup \Sigma)^*$ 
and $v$ is a power compressed word, namely $v=v_1^{k_1} \cdots v_d^{k_d}$
with binary encoded integers $k_j$ and $v_j \in (\Gamma \cup \Sigma)^*$.
We check whether $\sigma(u)=1$, which is an instance of $\PowWP(H)$, and make a case distinction:

\subparagraph{Case 1.} $\sigma(u)\neq 1$:
Since $H$ is torsion-free the equation $\sigma(u)^x = \sigma(v)$ has at most one solution.
We can solve it using the oracle for $\PowPP(H)$. 
If $\sigma(u)^x = \sigma(v)$ has no solution then also $u^x = v$ has no solution.
Otherwise we obtain a binary encoded $z \in \Z$ with $\sigma(u)^z = \sigma(v)$.
It remains to check whether $u^z = v$ in $A \wr H$, i.e. whether
$v_1^{k_1} \cdots v_d^{k_d} u^{-z} = 1$ in $A \wr H$.
This is an instance of $\PowWP(A \wr H)$.

\subparagraph{Case 2.} $\sigma(u)=1$.
We first check whether $\sigma(v)=1$ in $H$, which is an instance of $\PowWP(H)$.
If  $\sigma(v)\neq 1$ then we output {\bf no}.
Now assume that $\sigma(u)=\sigma(v)=1$.
We can compute $\supp(\tau(u))$ as well as $\Gamma$-words for all $\tau(u)(h)$ ($h \in \supp(\tau(u))$)
 in $\TC^0$ by going over all prefixes of the word $u$
(see Section~\ref{sec-wreath}).

Since $\sigma(u)=\sigma(v)=1$,
the equation $u^x = v$ is equivalent to 
$x \cdot \tau(u) = \tau(v)$. The f.g.~abelian group $A$ can be written as $A = \Z^m \times B$
for some finite abelian group $B$ and $m \in \N$. If $\tau(u)(h) \in B$ for all $h \in \supp(\tau(u))$
then $u$ has order at most $|B|$. Hence, 
there is a $z \in \Z$ with $u^z = v$ if and only if there is $0 \leq z < |B|$
with $u^z = v$. Using the oracle for $\PowWP(A \wr H)$ we can check
all such $z$ in parallel. Now assume that there is  $h \in \supp(\tau(u))$
such that $\tau(u)(h) = (\bm{a},b)$ for some $\bm{a} = \Z^m \setminus \{\bm{0}\}$.
From the word for $\tau(u)(h)$ we can compute the vector
$\bm{a}$ in unary notation. Moreover, 
using Lemma~\ref{lem:eval} we compute in $\TC^0$ a power word for
$\tau(v)(h) \in A$ in $\TC^0$. Let $\tau(v)(h) = (\bm{b},c)$. From the computed
power word we can compute (using simple arithmetic) the binary encoding of the vector
$\bm{b}$. 

Every solution $z$ for $u^x=v$ has to satisfy $z \cdot \bm{a} = \bm{b}$. 
The only candidate for this is $z = b_i/a_i$ (recall that integer division is in $\TC^0$) where $a_i$ is a non-zero entry 
of the vector $\bm{a} \neq \bm{0}$ and $b_i$ is the corresponding entry of $\bm{b}$.
If $z$ is not an integer, then $u^x=v$ has no solution. 
Otherwise, if $z \in \Z$, we check
whether $u^z = v$ using the oracle for $\PowWP(A \wr H)$. 
\end{proof}
  
\subsection{Proof of Theorem~\ref{thm:iterated-pwp}} \label{section-21->3}

In this section we deduce Theorem~\ref{thm:iterated-pwp}
 from Theorem~\ref{thm:ppp-wr}. Recall the definition of the iterated wreath products $W_{m,r}$. 
Every $W_{m,r}$ is orderable, and hence is torsion-free and has the unique roots property, see Section~\ref{sec-groups}.
The main point is that all groups $W_{m,r}$ are c-tame.
We start with the following easy lemma which covers the case $m=0$ (i.e., $W_{m,r} = \Z^r$).

\begin{lemma} \label{easy-integer-equations}
For vectors $\bm{a},\bm{b} \in \Z^n \setminus \{\bm{0}\}$ there exist numbers $s,t \in \Z$
with $|s| \leq \|\bm{b}\|$, $|t| \leq \|\bm{a}\|$ such that
$\{ (x,y) \in \Z^2 \mid x \bm{a} = y \bm{b} \} = \langle (s,t) \rangle$.
\end{lemma} 

\begin{proof}
If $\bm{a}$ and $\bm{b}$
are linearly independent over $\mathbb{Q}$
then $(0,0)$ is the only rational and hence integer solution.
Otherwise there exist coprime integers $s,t \in \Z \setminus \{0\}$ with $s \bm{a} = t \bm{b}$.
Since $s$ divides all entries in $\bm{b}$ we must have $|s| \le \|\bm{b}\|$,
and similarly for $t$ and $\bm{a}$.
Observe that $x \bm{a} = y \bm{b}$ is equivalent to $t x = s y$.

Therefore, every vector in $\langle (s,t)\rangle$ is a solution. Conversely,
if $tx=sy$, then $s$ divides $x$ (due to coprimality of $s$ and $t$) and thus
$(x,y)=\tfrac{x}{s}(s,t)$ is an integer multiple of $(s,t)$.
Hence the solution set is $\langle (s,t) \rangle$.
\end{proof}
Recall the definition of the iterated wreath products $W_{m,r}$. 
Every $W_{m,r}$ is orderable, and hence is torsion-free
and has the unique roots property, see Section~\ref{sec-groups}.

\begin{proposition} \label{prop:wmr-ctame}
For all $r \ge 1$, $m \ge 0$ the groups $W_{m,r}$ and $S_{m,r}$ are c-tame.
\end{proposition}

\begin{proof}
It suffices to show the statement for $W_{m,r}$.
We fix the rank $r$ and the prove the claim by induction on $m$.
If $m = 0$ then $W_{0,r} = \Z^r$ and the statement follows from Lemma~\ref{easy-integer-equations}.
Now assume that $m \ge 1$ and
let $u,v$ be words over the generators of $W_{m,r}$ with $u \neq 1 \neq v$ in $W_{m,r}$.
Let $U=\{(x,y) \in \Z^2 \mid u^x =v^y \}$ and
$V=\{(x,y) \in \Z^2 \mid \sigma(u)^x=\sigma(v)^y\}$.
By Lemma~\ref{lemma-2dim} the sets $U,V$ are subgroups of $\Z^2$ and $U \leq V$.

\subparagraph{Case 1.} $\sigma(u) \neq 1 \neq \sigma(v)$:
Then $V$ is cyclic by Lemma~\ref{lemma-2dim},
say $V= \langle (s, t) \rangle$.
By the induction hypothesis, $|s|$ and $|t|$ are polynomially bounded in $|\sigma(uv)| \leq |uv|$.
Since $U \leq V$ we can write $U$ as $U = \langle a\cdot (s, t) \rangle$ for some $a>0$. 
We obtain the identity $u^{as}=v^{at}$ in $W_{m,r}$.
Since $W_{m,r}$ has the unique roots property we get
$u^{s}=v^{t}$ and hence $U=V = \langle (s,t) \rangle$.
Since  $|s|$ and $|t|$ are polynomially bounded in $|uv|$, we are done.

\subparagraph{Case 2.} $\sigma(u)=1$ or $\sigma(v)=1$ in $W_{m-1,r}$.
If exactly one of these projections is $1$, then we only have the trivial solution $(0,0)$ for the equation $u^x=v^y$,
since $W_{m-1,r}$ and $W_{m,r}$ are torsion-free.
If $\sigma(u)= \sigma(v)=1$ then $\tau(u) \neq 0 \neq \tau(v)$ by $u \neq 1 \neq v$.
The equation $u^x =v^y$ is equivalent to $x \cdot \tau(u) = y \cdot \tau(v)$
in the abelian group $(\Z^r)^{(W_{m-1,r})}$.
Since the absolute values of the integers that appear in the images of $\tau(u)$ and $\tau(v)$
are linearly bounded by $|uv|$ we can conclude the proof with Lemma~\ref{easy-integer-equations}.
\end{proof}
We can now show Theorem~\ref{thm:iterated-pwp}:

\begin{proof}[Proof of Theorem~\ref{thm:iterated-pwp}]
We will prove by induction on $m \in \N$
that $\PowPP(W_{m,r})$ and hence also $\PowWP(W_{m,r})$ belongs to $\TC^0$.
If $m = 0$ then $\PowPP(W_{0,r})$ is the problem of solving a system
of $r$ linear equations $a_i x = b_i$ where $a_i$ is given in unary encoding
and $b_i$ is given in binary encoding for $1 \le i \le r$.
Since integer division belongs to $\TC^0$
(here, we only have to divide by the unary encoded integers $a_i$) this problem can be solved
in $\TC^0$.
The inductive step follows from Theorem~\ref{thm:ppp-wr} and the fact that all groups $W_{m,r}$
are c-tame (Proposition~\ref{prop:wmr-ctame}) and torsion-free.
\end{proof}

 \subsection{Proof of Theorem~\ref{thm:abelian-wr}}
 
 For the rest of this section fix the groups $H$ and $A$ from Theorem~\ref{thm:abelian-wr}. Hence, $A$ is f.g.~abelian
and $H$ is orderable and knapsack-semilinear with $\mathsf{E}_H(n) = 2^{n^{\O(1)}}$. By Theorem~\ref{thm:closure-wreath}, 
also $A \wr H$ is knapsack-semilinear.

The main idea for the proof of Theorem~\ref{thm:abelian-wr} is to describe the solution set $\Sol_{A \wr H}(E)$ for a given exponential expression
$E$ by a Presburger formula (Section~\ref{sec-constructing-formula}). This formula is an exponentially long 
disjunction of existential Presburger formulas. For bounding the magnitude 
of the solution set, the disjunction (leading to a union of semilinear sets) as well as the existential quantifiers (leading to a projection
of a semilinear set) have no influence. The remaining formula is a polynomially large conjunction of semilinear constraints of 
exponential magnitude. With Lemma~\ref{lem:cap} we then obtain an exponential bound on the magnitude of the solution set.

A crucial fact is that our Presburger formula for $\Sol_{A \wr H}(E)$ does not involve quantifier alternations. This is in contrast to 
the Presburger formulas constructed in \cite{GanardiKLZ18} for showing that the class of knapsack-semilinear groups is closed under wreath products.
We can avoid quantifier alternations since we restrict to wreath products $A \wr H$ with $A$ abelian.
Let us also remark that we do not have to algorithmically construct the Presburger formula for the solution set. Only its existence is 
important, which yields an exponential bound on the size of a solution. 

Before we construct the Presburger formula for the set of solutions, we first have to introduce a 
certain decomposition of solutions that culminates in Proposition~\ref{prop:main}.

\subsubsection{Decomposition into packed bundles}

In this section, we will only work with the orderable group $H$.
A {\em bundle} $P$ is a finite multiset of progressions over $H$.
A {\em refinement} of a progression $\bm{p} = (p_i)_{0 \le i \le m}$
is a bundle 
$\ms{ (p_i)_{m_{k-1} \le i \le m_k-1} \mid 1 \le k \le \ell }$
for some $0 = m_0 < m_1 < \dots < m_\ell = m+1$.
A bundle $Q$ is a {\em refinement} of a bundle $P$
if one can decompose $Q = \bigcup_{\bm{p} \in P} Q_{\bm{p}}$ such that each $Q_{\bm{p}}$ is a refinement of $\bm{p}$.
We emphasize that a union of bundles is always understood as the union of multisets, and that $|Q|$  (for a bundle $Q$)
refers to the size of $Q$ as a multiset.

Two progressions $\bm{p}_1, \bm{p}_2$ are {\em disjoint} if $\supp(\bm{p}_1) \cap \supp(\bm{p}_2) = \emptyset$.
Two bundles $P,Q$ are {\em disjoint} if any two progressions $\bm{p} \in P$, $\bm{q} \in Q$ are disjoint.
A bundle $P$ is {\em stacking} if there exists $h \in H$
such that $\supp(\bm{p}) = \{h\}$ for all $\bm{p} \in P$.

\begin{lemma}
	\label{lem:decomp}
	For every bundle $P$
	there exists a refinement $Q$ of $P$ of size $|Q| = \O(|P|^3)$
	and a partition $Q = \bigcup_k Q_k$ into pairwise disjoint subbundles $Q_k$ such that
	each bundle $Q_k$ consists of parallel rays or is stacking.
\end{lemma}

\begin{proof}
	Let $S$ be the union of all intersections $\supp(\bm{p}) \cap \supp(\bm{q})$
	of size one over all $\bm{p},\bm{q} \in P$, which contains at most $|P|^2$ elements.
	We refine each progression $\bm{p} = (p_i)_{0 \le i \le m}$ into progressions $\bm{p}^{(j)}$ such that
	$|\supp(\bm{p}^{(j)})| = 1$ or $\supp(\bm{p}^{(j)}) \cap S = \emptyset$ as follows.
	Define the following relation on $[0,m]$:
	Let $i_1 \sim i_2$ if either (i) there exists $h \in S$ such that $p_i = h$ for all $i_1 \le i \le i_2$
	or (ii) $p_i \notin S$ for all $i_1 \le i \le i_2$.
	Notice that this defines an equivalence relation,
	which partitions $[0,m]$ into at most $2|S|+1 \le 2|P|^2 + 1$ many intervals
	and in that way yields a refinement $Q_{\bm{p}}$ of $\bm{p}$ of size $2|P|^2 + 1$.
	Let $Q$ be the union of all bundles $Q_{\bm{p}}$ over all $\bm{p} \in P$,
	which contains $\O(|P|^3)$ many progressions.
	Notice that $S$ is still the union of all intersections $\supp(\bm{p}) \cap \supp(\bm{q})$
	of size one over all $\bm{p}, \bm{q} \in Q$.
	Therefore any two progressions $\bm{p}, \bm{q} \in Q$ with
	$|\supp(\bm{p}) \cap \supp(\bm{q})| = 1$ satisfy
	$|\supp(\bm{p})| = |\supp(\bm{q})| = 1$ and $\supp(\bm{p}) = \supp(\bm{q})$.
	
	Finally we define the subbundles $Q_k$.
	Two $\bm{p},\bm{q} \in Q$ are bundled together if
	\begin{enumerate}
	\item $|\supp(\bm{p})|,|\supp(\bm{q})| = 1$ and $\supp(\bm{p}) = \supp(\bm{q})$, or
	\item $|\supp(\bm{p})|,|\supp(\bm{q})| \ge 2$ and $\bm{p},\bm{q}$ are parallel.
	\end{enumerate}
	Let us verify that any two progressions $\bm{p}, \bm{q} \in Q$ which are not in the same bundle have disjoint supports.
	As observed above, if $|\supp(\bm{p}) \cap \supp(\bm{q})| = 1$
	then $\supp(\bm{p}) = \supp(\bm{q})$, which would mean that $\bm{p}, \bm{q}$ are in the same bundle.
	If $|\supp(\bm{p}) \cap \supp(\bm{q})| \ge 2$ then $\bm{p}$ and $\bm{q}$ are parallel rays by Lemma~\ref{lem:rays-cap},
	which contradicts the fact that they are in different bundles.
\end{proof}

A ray $\bm{p} = (ab^i)_{0 \le i \le m}$ is {\em packed} into a ray $\bm{q} = (gh^j)_{0 \le i \le \ell}$
if $b = h^d$ for some $d \in \Z \setminus \{0\}$
and $\supp(\bm{p}) = \supp(\bm{q}) \cap a\langle b \rangle$.
Intuitively, this means that $\bm{p}$ is contained in $\bm{q}$
and $\bm{p}$ cannot be extended in $\bm{q}$.
More explicitly, the latter condition states that
$i \in \Z$ and $ab^i \in \supp(\bm{q})$ implies $i \in [0,m]$
(we call this the {\em maximality condition}).
A bundle $P$ of rays is {\em packed} into $\bm{q}$ if every $\bm{p} \in P$
is packed into $\bm{q}$.

\begin{lemma}
\label{lem:pq}
Let $\bm{p} = (ab^i)_{0 \le i \le m}$, $\bm{q} = (gh^j)_{0 \le j \le \ell}$ be rays
with $b = h^d$ for some $d \in \Z \setminus \{0\}$.
\begin{enumerate}[(i)]
\item $\{ i \in \Z \mid ab^i \in \supp(\bm{q}) \}$ is an interval. \label{lem:is-interval}
\item $\bm{p}$ is packed into the ray $\bm{q}$ if and only if
$a,ab^m \in \supp(\bm{q})$ and
$ab^{-1}, ab^{m+1} \in \overline{\supp}(\bm{q})$
where $\overline{\supp}(\bm{q}) = \{ gh^j \mid j \in \Z \setminus [0,\ell] \}$. \label{lem:packed}
\item If a bundle $P$ is packed into $\bm{q}$
then $P$ is packed into a subray $\bm{q'}$ of $\bm{q}$
whose endpoints are endpoints of rays in $P$. \label{lem:endpoints}
\item If $\bm{p}$ is packed into $\bm{q}$
then $\supp(\bm{p}) = \{ gh^j \mid 0 \le j \le \ell, \, j \equiv t \pmod d \}$
for some unique remainder $0 \le t < |d|$. \label{lem:mod}
\end{enumerate}
\end{lemma}

\begin{proof}
For point \eqref{lem:is-interval} consider integers $i_1 \le i \le i_2$
and assume that $ab^{i_1}, ab^{i_2} \in \supp(\bm{q})$, 
i.e. there exist $j_1,j_2 \in [0,\ell]$ with $ab^{i_1} = gh^{j_1}$ and $ab^{i_2} = gh^{j_2}$.
Hence $ah^{di_1} = gh^{j_1}$ and $ah^{di_2} = gh^{j_2}$. From this we obtain
\[ h^{j_2-j_1} = (gh^{j_1})^{-1} gh^{j_2}=(ah^{di_1})^{-1}ah^{di_2}=  h^{d(i_2-i_1)} \]
and therefore $j_2-j_1 = d(i_2-i_1)$ (since $h$ has infinite order).
We claim that $ab^i = ab^{i_1 + (i-i_1)} = gh^{j_1 + d(i-i_1)}$ belongs to $\supp(\bm{q})$:
If $d>0$ then $j_1 \le j_1 + d(i-i_1) \le j_1+d(i_2-i_1) = j_2$,
thus, $ab^i \in \supp(\bm{q})$. 
The case $d<0$ is symmetric.

For point \eqref{lem:packed},
if $\bm{p}$ is packed into $\bm{q}$ then $a \in \supp(\bm{q})$ by definition,
i.e. $a = gh^j$ for some $j \in [0,\ell]$.
Therefore $ab^{-1} = gh^jb^{-1} = gh^{j-d}$ and, since $ab^{-1} \notin \supp(\bm{q})$,
we deduce that $j-d \notin [0,\ell]$ by the maximality condition.
Similarly $ab^m \in \supp(\bm{q})$ by definition, i.e. $ab^m = gh^{j'}$
for some $j' \in [0,\ell]$.
Therefore $ab^{m+1} = gh^{j'+d}$ and, since $ab^{m+1} \notin \supp(\bm{q})$
we know that $j'+d \notin [0,\ell]$.

For the direction from right to left assume that $a,ab^m \in \supp(\bm{q})$
and $ab^{-1},ab^{m+1} \in \overline{\supp}(\bm{q})$.
From point~\eqref{lem:is-interval} we get $\supp(\bm{p}) \subseteq \supp(\bm{q})$.
Moreover, if $ab^i \in \supp(\bm{q})$ for some $i \in \Z \setminus [0,m]$
then point~\eqref{lem:is-interval} would imply $ab^{-1} \in \supp(\bm{q})$ or $ab^{m+1} \in \supp(\bm{q})$,
which is a contradiction.

For point \eqref{lem:endpoints}
suppose that $q$ is an endpoint of $\bm{q}$ which is not the endpoint of any ray $\bm{p} \in P$.
If $\bm{q'}$ is obtained by removing $q$ from $\bm{q}$ then
the property from point~\eqref{lem:packed} is preserved for $\bm{q'}$
since $\overline{\supp}(\bm{q}) \subseteq \overline{\supp}(\bm{q'})$.
Hence we can remove endpoints of $\bm{q}$ until the desired property is satisfied.

For point \eqref{lem:mod} assume that $\bm{p}$ is packed into $\bm{q}$. Hence, we have
$\supp(\bm{p}) = \supp(\bm{q}) \cap a\langle b \rangle = \supp(\bm{q}) \cap a\langle h^d \rangle$. 
There exists $s \in [0,m]$ with $a = gh ^s$. Let $t = s \bmod d$. 
It suffices to show 
\[
\{ gh^j \mid 0 \le j \le \ell, \, j \equiv t \!\!\pmod d \} = \{ gh^j \mid 0 \le j \le \ell \} \cap a\langle h^d \rangle .
\]
First, consider some $j \in [0,\ell]$ with $j \equiv t \pmod d$. We have to show that $gh^j \in a\langle h^d \rangle$.
Since $j \equiv t \pmod d$ we have $j \equiv s \pmod d$. Let $j = s + r d$ for some $r \in \Z$.
We obtain $gh^j = gh^{s+rd} = gh^s h^{rd} = a h^{rd}$.

For the other inclusion let $j \in [0,\ell]$ such that $gh^j \in a\langle h^d \rangle$, i.e.
$gh^j = a h^{d i}$ for some $i \in \Z$. We have to show that $j \equiv t \pmod d$.
Since $a = gh^s$ we have
$gh^j = ah^{di}=g h^{s+d i}$, i.e. $j = s+di$. Hence, $j \equiv s \pmod d$, and therefore $j \equiv t \pmod d$.

The remainder is clearly unique since $h$ is nontrivial.
\end{proof}

\begin{lemma}
\label{lem:h-packed}
Let $h \in H$ and let $P$ be a bundle of parallel rays whose periods are contained in $\langle h \rangle$.
Then there exist a refinement $Q$ of $P$ of size $|Q| = \O(|P|^2)$
and a partition $Q = \bigcup_k Q_k$
into pairwise disjoint subbundles $Q_k$
such that each subbundle $Q_k$ is packed into a ray with period $h$.
\end{lemma}

\begin{proof}
For every ray $\bm{p} \in P$ there exists a left coset $g\langle h \rangle$ which contains $\supp(\bm{p})$.
Therefore we can split $P$ into disjoint bundles $P_{g \langle h \rangle} = \{ \bm{p} \in P \mid \supp(\bm{p}) \subseteq g \langle h \rangle \}$
and treat each bundle $P_{g \langle h \rangle}$ individually.

Consider a left coset $K$ of $\langle h \rangle$ in $H$
and suppose that $\supp(\bm{p}) \subseteq K$ for all $\bm{p} \in P$.
Define the linear order $\le_h$ on $K$ by $h_1 \le_h h_2$ if $h_1 h^d = h_2$ for some $d \in \N$.
Define $\beta \colon K \to 2^P$ by
\[
	\beta(g) = \{ \bm{p} \in P \mid \exists g_1,g_2 \in \supp(\bm{p}) \colon g_1 \le_h g \le_h g_2 \}.
\]
Intuitively, $\beta(g)$ contains all rays $\bm{p}$ that cover the element $g$.
The mapping $\beta$ satisfies the condition of Lemma~\ref{lem:interval-types}
and hence we obtain a partition $\mathcal{J} = \{J_1, \dots, J_k \}$ of $K$ into at most $\O(|P|)$
many intervals (with respect to $\leq_h$)
and subsets $P_J \subseteq P$
such that $\beta(J) = \{ P_J \}$ for all $J \in \mathcal{J}$.

For $\bm{p} \in P$ and $J \in \mathcal{J}$ define the restriction
$\bm{p}|_J$ to those entries $p_i \in J$.
Notice that, if $\bm{p}|_J$ is non-empty, then it is a subray of $\bm{p}$
since the natural order on $\bm{p}$ respects $\le_h$ or $\ge_h$,
i.e. either $i \le j$ implies $p_i \le_h p_j$ or it implies $p_i \ge_h p_j$,
depending on whether the period of $\bm{p}$ is a positive or a negative power of $h$.
Furthermore, if $\bm{p} \in P \setminus P_J$ then $\bm{p} \notin \beta(g)$ for all $g \in J$
and thus $\bm{p}|_J$ is empty.

For every $J \in \mathcal{J}$ let $Q_J$ be the bundle containing all non-empty restrictions
$\bm{p}|_J$ for $\bm{p} \in P_J$.
Then $Q = \bigcup_{J \in \mathcal{J}} Q_J$ is a refinement of $P$ and the subbundles $Q_J$
are pairwise disjoint.
Its size is bounded by $|Q| \le |P||\mathcal{J}| = \O(|P|^2)$.
It remains to prove that every bundle $Q_J$ is packed into a ray with period $h$. 
Consider an interval $J \in \mathcal{J}$.
If $P_J = \emptyset$ then $Q_J$ is empty and the claim is vacuous.
If $P_J$ contains some ray $\bm{p'}$
then for all $g \in J$ there exist $g_1,g_2 \in \supp(\bm{p'})$ with $g_1 \le_h g \le_h g_2$.
Since $\supp(\bm{p'})$ is finite also $J$ must be finite.
Therefore we can write $J = \{ gh^j \mid 0 \le j \le \ell \}$ for some $g \in J$ and $\ell \in \N$.
We naturally view $J$ as the ray $\bm{q}_J = (gh^j)_{0 \le j \le \ell}$.
We claim that for all $\bm{p} \in P_J$ the restriction $\bm{p}|_J$ is packed into $\bm{q}_J$.

Suppose that $\bm{p} = (ab^i)_{0 \le i \le m}$
and that $\bm{p}|_J = (ab^i)_{s \le i \le t}$ for some $0 \le s \le t \le m$.
First observe that $b$ is a power of $h$, say $b = h^d$ for $d \in \Z \setminus \{0\}$,
and that $\supp(\bm{p}|_J) \subseteq J = \supp(\bm{q}_J)$.
Let $ab^i \in \supp(\bm{q}_J) = J$ be an arbitrary element with $i \in \Z$,
and thus $\bm{p} \in P_J = \beta(ab^i)$.
It follows that there exist $ab^{i_1},ab^{i_2} \in \supp(\bm{p})$ with
$0 \le i_1,i_2 \le m$ and $ab^{i_1} \le_h ab^i \le_h ab^{i_2}$.
Therefore either $ab^{i_1} \le_b ab^i \le_b ab^{i_2}$
or $ab^{i_1} \ge_b ab^i \ge_b ab^{i_2}$.
This implies $ab^i \in \supp(\bm{p})$ and thus $ab^i \in \supp(\bm{p}|_J)$.
This concludes the proof.
\end{proof}

\subsubsection{From knapsack to bundles} \label{sec:knapsack-bundles}

Fix a normalized exponent expression $E = u_1^{x_1} v_1 u_2^{x_2} v_2 \cdots u_d^{x_d} v_d$
over $A \wr H$ for the rest of this section where $|E| \le n$. Let $X  = \{ x_1, \ldots, x_d\}$ be the set of variables
appearing in $E$. 
Since $u_1, \dots, u_d \in AH$ there exist (unique) elements $a_1, \dots, a_d \in A$ such that $u_r = a_r \, \sigma(u_r)$
for all $1 \le r \le d$. For $1 \le r \le d$ and a fresh variable $y \notin X$ we define the exponent expression
\begin{equation}
	\label{def:sigma-ry}
	E_r(y) = u_1^{x_1} v_1 \cdots u_{r-1}^{x_{r-1}} v_{r-1} u_r^y.
\end{equation}
Let $1 \le r \le d$, $\nu \in \N^X$ and $k \in \N$. With $\nu_{[y/k]}$ we denote the valuation that
extends $\nu$ by $\nu_{[y/k]}(y) = k$.
We define $\sigma_{\nu}(r,k) = \nu_{[y/k]}(\sigma(E_r(y)))$ 
and $\tau_{\nu}(r,k) = \tau(\sigma_{\nu}(r,k) \, a_r)$.
Given $0 \le s \le t \le \nu(x_r)-1$ we define
$\sigma_{\nu}(r,s,t) = (\sigma_{\nu}(r,k))_{s \le k \le t}$
and $\tau_{\nu}(r,s,t) = \sum_{k = s}^t \tau_{\nu}(r,k)$.
Notice that $\sigma_{\nu}(r,s,t)$ is a progression with period $\sigma(u_r)$ by \eqref{def:sigma-ry}.
Furthermore we have
\begin{equation}
\begin{aligned}
\label{eq:rst-supp}
	\supp(\tau_{\nu}(r,s,t)) = \ & \supp\Big(\sum_{k = s}^t \tau_{\nu}(r,k)\Big)
	\subseteq \bigcup_{s \le k \le t} \supp( \tau_{\nu}(r,k)) \\
	 \subseteq \ & \{\sigma_{\nu}(r,k)) \mid s \le k \le t \} = \supp(\sigma_{\nu}(r,s,t)).
\end{aligned}
\end{equation}
If $r \in R$, i.e. $\sigma(u_r) \neq 1$, and $h \in \supp(\sigma_{\nu}(r,s,t))$
then there exists exactly one index $s \le k_h \le t$ such that $h = \sigma_{\nu}(r,k_h)$ and
\begin{equation}
\begin{aligned}
\label{eq:shift-sum}
	\tau_{\nu}(r,s,t)(h) = \sum_{k = s}^t \tau_{\nu}(r,k)(\sigma_{\nu}(r,k_h))
	= \tau_{\nu}(r,k_h)(\sigma_{\nu}(r,k_h)) = a_r.
\end{aligned}
\end{equation}
For a valuation $\nu \in \N^X$ we define a {\em $\nu$-decomposition}
to be a set $D \subseteq [1,d] \times \N^2$ such that
$\{ \{r\} \times [s,t] \mid (r,s,t) \in D \}$
is a partition of $\{ (r,k) \mid 1 \le r \le d, \, 0 \le k \le \nu(x_r)-1 \}$.

\begin{lemma}
	\label{lem:nu-decomp}
	For all $\nu \in \N^X$ and $\nu$-decompositions $D$ we have $\tau(\nu(E)) = \sum_{(r,s,t) \in D} \tau_{\nu}(r,s,t)$.
\end{lemma}

\begin{proof}
First we observe that
\[
\sum_{(r,s,t) \in D} \tau_{\nu}(r,s,t) = \sum_{(r,s,t) \in D} \sum_{s \le k \le t} \tau_{\nu}(r,k) = \sum_{1 \le r \le d} \sum_{0 \le k \le \nu(x_r)-1} \tau_{\nu}(r,k).
\]
For all $1 \le r \le d$ and $0 \le k \le \nu(x_r)-1$ we have
\begin{align*}
	\tau_{\nu}(r,k) &= \tau(\sigma_{\nu}(r,k) \, a_r) \\
	&= \tau(\sigma(u_1^{\nu(x_1)} v_1 \cdots u_{r-1}^{\nu(x_{r-1})} v_{r-1} u_r^k) \, a_r) \\
	&= \tau(\sigma(u_1^{\nu(x_1)} v_1 \cdots u_{r-1}^{\nu(x_{r-1})} v_{r-1} u_r^k) \, u_r)
\end{align*}
where the last equality follows from $u_r = a_r \, \sigma(u_r)$.
Then the statement follows easily from Lemma~\ref{lem:tau}.
\end{proof}
Let $R = \{ r \in [1,d] \mid \sigma(u_r) \neq 1 \}$
and define the equivalence relation $\parallel$ on $R$ by $r_1 \parallel r_2$
if $\sigma(u_{r_1})$ and $\sigma(u_{r_2})$ are commensurable.

\begin{lemma}
\label{lem:gen}
Let $C$ be a $\parallel$-class and
let $U = \{\sigma(u_r) \mid r \in C \}$.
Then there exist numbers $\alpha_{r},\beta_{r} \in \Z$ for $r \in C$ with $|\alpha_{r}|, |\beta_{r}| \leq 
(|C|+1) \cdot \mathsf{K}_H(n)^n$ and an element $h_C \in H$ such that
$h_C = \prod_{r \in C} \sigma(u_{r})^{\alpha_{r}}$ and $h_C^{\beta_{r}} = \sigma(u_{r})$ for all $r \in C$.
\end{lemma}

\begin{proof}
By Lemma~\ref{lemma-commensurable-cyclic}, the group $\langle U \rangle$ is cyclic of infinite order.
Let $\psi \colon \langle U \rangle \to \Z$ be any isomorphism.
First we ensure that all numbers $|\psi(\sigma(u_r))|$ are bounded exponentially.
Fix an element $s \in C$ and let $C' = C \setminus \{s\}$.
Because of commensurability for each $r \in C$ there exist numbers $p_{r},q_{r} \in \Z \setminus \{0\}$
such that $p_{r} \cdot \psi(\sigma(u_{s})) = q_{r} \cdot \psi(\sigma(u_{r}))$ (we can take $p_{s} = q_{s}=1$).
We can assume that the numbers $p_{r}$ and $q_{r}$ are coprime for all $r \in C$.
Let $\lambda\in\N$ be the least common multiple of the numbers $q_r$ for $r\in C'$.
Then $\psi(\sigma(u_{s}))$ is divided by $\lambda$,
say $\psi(\sigma(u_{s})) = \delta \cdot \lambda$.
In fact every number $\psi(\sigma(u_{r}))$ is divided by $\delta$
since we have
\[
	p_{r} \cdot \delta \cdot \lambda / q_{r} = \psi(\sigma(u_{r}))
\]
for $r \in C'$.
Hence every number in $\psi(\langle U \rangle) = \Z$ is divided by $\delta$
which implies $|\delta| = 1$.
Since $\sum_{r \in C} |\sigma(u_{r})|$ is bounded by $n$ we can further assume that
$|p_{r}|$ and $|q_{r}|$ are bounded by $\mathsf{K}_H(n)$ for all $r \in C$.
Let us define $\beta_{r} := \psi(\sigma(u_{r}))$ for all $r \in C$. We get
\[
	|\beta_{r}| = |\psi(\sigma(u_{r}))| \le |p_{r}| \cdot \lambda \le \mathsf{K}_H(n)^n.
\]
Since $1 \in \psi(\langle U \rangle)$ there exist numbers $\alpha_{r} \in \Z$ for $r \in C$ such that
\[
	\psi \left(\prod_{r \in C} \sigma(u_{r})^{\alpha_{r}}\right) = \sum_{r \in C} \alpha_{r} \cdot \psi(\sigma(u_{r})) = 1.
\]
By the standard bounds \cite{vZGS78} there exists such a solution where
\[|\alpha_{r}| \le (|C|+1) \cdot \max \{ |\psi(\sigma(u_{r}))| \mid r \in C \} \le (|C|+1) \cdot \mathsf{K}_H(n)^n.\]
Finally we set $h_C = \psi^{-1}(1)$.
\end{proof}
By Lemma~\ref{lem:gen} there exist numbers $\alpha_r,\beta_r \in \Z$ for $r \in R$
and elements $h_C \in H$ for all $\parallel$-classes $C$ such that the following holds
(recall that by assumption $\mathsf{E}_H(n)$ is exponentially bounded):
\begin{itemize}
\item $|\alpha_r|,|\beta_r| \le 2^{n^{\O(1)}}$ for all $r \in R$,
\item $h_C = \prod_{r \in C} \sigma(u_r)^{\alpha_r}$ for all $\parallel$-classes $C$,
\item $h_C^{\beta_r} = \sigma(u_r)$ for all $\parallel$-classes $C$ and $r \in C$.
\end{itemize}

\begin{proposition}
\label{prop:main}
A valuation $\nu \in \N^X$ satisfies
$\nu(E) = 1$ if and only if $\nu(\sigma(E)) = 1$ and
there exists a $\nu$-decomposition $D$ of size $\O(n^6)$ and a partition $\{ D_1, \dots, D_m \}$ of $D$
such that for all $1 \le i \le m$ we have:
\begin{itemize}
\item $\sum_{(r,s,t) \in D_i} \tau_{\nu}(r,s,t) = 0$ and
\item the bundle $Q_i = \ms { \sigma_{\nu}(r,s,t) \mid (r,s,t) \in D_i }$ is stacking
or $Q_i$ is a bundle of parallel rays
that is packed into a ray with period $h_C$ for some $\parallel$-class $C$.
\end{itemize}
\end{proposition}

\begin{proof}
By Lemma~\ref{lem:nu-decomp} we know
\begin{equation}
\label{eq:d-i-sum}
\tau(\nu(E)) = \sum_{(r,s,t) \in D} \tau_{\nu}(r,s,t) = \sum_{1 \le i \le m} \sum_{(r,s,t) \in D_i} \tau_{\nu}(r,s,t)
\end{equation}
since $\{D_1, \dots, D_m\}$ forms a partition of $D$.
Then the direction from right to left is easy
since \eqref{eq:d-i-sum} implies $\nu(\tau(E)) = 0$,
and together with $\nu(\sigma(E)) = 1$ we get $\nu(E) = 1$.

Conversely, if $\nu(E) = 1$ then clearly $\nu(\sigma(E)) = 1$.
Let us define the bundle
\[ P = \ms{\sigma_{\nu}(r,1,\nu(x_r)) \mid 1 \le r \le d }.
\]
The period of $\sigma_{\nu}(r,1,\nu(x_r))$ is $\sigma(u_r)$,
and if $r \in R$, then $\sigma(u_r) \in \langle h_{[r]} \rangle$, where $[r]$ denotes the $\parallel$-class of $r$.
By Lemma~\ref{lem:decomp} and Lemma~\ref{lem:h-packed}
there exists a refinement $Q$ of $P$ of size $\O(d^6) \le \O(n^6)$ and 
a partition $Q = \bigcup_{i=1}^m Q_i$ into pairwise disjoint subbundles $Q_i$ such that
each bundle $Q_i$ is stacking or is packed into a ray with period $h_C$ for some $\parallel$-class $C$.
The bundles $Q$ and $Q_1, \dots, Q_m$ induce a $\nu$-decomposition $D$ and a partition
$\{D_1, \dots, D_m\}$ of $D$ such that
\[ Q_i = \ms{ \sigma_{\nu}(r,s,t) \mid (r,s,t) \in D_i }\] 
for all $1 \le i \le m$.
By \eqref{eq:d-i-sum} we can derive
$\sum_{i=1}^m \sum_{(r,s,t) \in D_i} \tau_{\nu}(r,s,t) = 0$.
We claim that the summands $\sum_{(r,s,t) \in D_i} \tau_{\nu}(r,s,t)$ have disjoint supports
and thus each summand must be equal to 0.
Observe that
\[
	\supp\Big(\sum_{(r,s,t) \in D_i }\tau_{\nu}(r,s,t)\Big) \subseteq \bigcup_{(r,s,t) \in D_i} \supp(\tau_{\nu}(r,s,t))
	\subseteq \bigcup_{(r,s,t) \in D_i} \supp(\sigma_{\nu}(r,s,t))
\]
where the second inclusion follows from \eqref{eq:rst-supp}.
The claim follows from the fact that the subbundles $Q_i$ are pairwise disjoint.
\end{proof}

\subsubsection{Constructing the formulas} \label{sec-constructing-formula}

A {\em bundle descriptor} is a set $\theta = \{(r_1, y_1, z_1), \dots, (r_m, y_m, z_m)\}$
where $1 \le r_i \le d$ for all $1 \le i \le m$ and $y_1,z_1, \dots, y_m,z_m \notin X$ are $2m$ distinct fresh variables.
We define $V_\theta = \{y_1,z_1,\dots,y_m,z_m\}$
and the extended set of variables $X_\theta = X \cup V_\theta$.
A {\em $\theta$-valuation} is a valuation $\nu \in \N^{X_\theta}$ such that
$\nu(y_i) \le \nu(z_i) \le \nu(x_{r_i})$ for all $1 \le i \le m$.
We will use the numbers $\alpha_r$ and $\beta_r$ ($r \in R$) constructed in the previous subsection.

\begin{lemma}
\label{lem:stack}
Let $\theta = \{(r_1, y_1, z_1), \dots, (r_m, y_m, z_m)\}$ be a bundle descriptor.
There exists an existential Presburger formula $\mathsf{Stack}_\theta$
with free variables over $X_\theta$ such that a $\theta$-valuation $\nu \in \N^{X_\theta}$
satisfies $\mathsf{Stack}_\theta$ if and only if
\begin{enumerate}
\item The bundle $\ms{ \sigma_\nu(r_i,\nu(y_i),\nu(z_i)) \mid 1 \le i \le m }$ is stacking, and
\item $\sum_{i=1}^m \tau_\nu(r_i,\nu(y_i),\nu(z_i)) = 0$.
\end{enumerate}
Furthermore $\mathsf{Stack}_\theta$ defines a semilinear set with magnitude $2^{(n+m)^{\O(1)}}$.
\end{lemma}

\begin{proof}
Let $\nu$ be an $\theta$-valuation. For better readability we define $s_i = \nu(y_i)$
and $t_i = \nu(z_i)$ for $1 \le i \le m$.
By definition $\ms{ \sigma_\nu(r_i,s_i,t_i) \mid 1 \le i \le m }$ is stacking
if and only if there exists $h \in H$ such that for all $1 \le i \le m$ we have
$\{ \sigma_\nu(r_i,k) \mid s_i \le k \le t_i \} = \{h\}$.
Since $\sigma_\nu(r_i,k) = \sigma_\nu(r_i,s_i) \, \sigma(u_{r_i})^{k-s_i}$,
this is equivalent to the statement that for each $1 \le i \le m$,
we have (i)~$\sigma_\nu(r_1,s_1) = \sigma_\nu(r_i,s_i)$
and $s_i = t_i$ or (ii)~$\sigma(u_{r_i}) = 1$, i.e. $r_i \notin R$.
Under condition~1.~from the lemma, condition 2.~is equivalent to
\[
	0 = \sum_{i=1}^m \tau_\nu(r_i,s_i,t_i)(h) = \sum_{i=1}^m \sum_{k=s_i}^{t_i} \tau_\nu(r_i,k)(h) 
	= \sum_{i=1}^m \sum_{k=s_i}^{t_i} a_{r_i} = \sum_{i=1}^m (t_i-s_i+1) \cdot a_{r_i},
\]
where $h \in H$ is the unique element with $\sigma_\nu(r_i,k) = h$ for all $1 \le i \le m$, $s_i \le k \le t_i$.
This description can be directly expressed as the following formula (we use the exponent expressions from \eqref{def:sigma-ry}):
\begin{equation}
	\mathsf{Stack}_\theta = \bigwedge_{i=1}^m (\sigma(E_{r_1}(y_1)) = \sigma(E_{r_i}(y_i)) \wedge (y_i = z_i \vee r_i \notin R)) \wedge \sum_{i=1}^m (z_i-y_i+1) \cdot a_{r_i} = 0.
\end{equation}
It consists of $m$ exponent equations over $H$ of length $\O(n)$,
identities between variables, 
and an exponent equation over $A$ of length $\O(mn)$.
By Lemma~\ref{lem:abelian} and Lemma~\ref{lem:cap}
the semilinear set defined by $\mathsf{Stack}_\theta$ has magnitude $2^{(n+m)^{\O(1)}}$.
\end{proof}

\begin{lemma}
\label{lem:pack}
Let $\theta = \{(r_1, y_1, z_1), \dots, (r_m, y_m, z_m)\}$ be a bundle descriptor
such that $r_1, \dots, r_m \in C$ for some $\parallel$-class $C$.
There exists an existential Presburger formula $\mathsf{Pack}_\theta$
with free variables over $X_\theta$ such that a $\theta$-valuation $\nu \in \N^{X_\theta}$
satisfies $\mathsf{Pack}_\theta$ if and only if
there exists a ray $\bm{q}$ with period $h_C$ such that
\begin{enumerate}
\item The bundle $\ms{ \sigma_{\nu}(r_i,\nu(y_i),\nu(z_i)) \mid 1 \le i \le m }$ is packed into $\bm{q}$, and
\item $\sum_{i=1}^m \tau_{\nu}(r_i,\nu(y_i),\nu(z_i)) = 0$.
\end{enumerate}
Furthermore $\mathsf{Pack}_\theta$ defines a semilinear set with magnitude $2^{(n+m)^{\O(1)}}$.
\end{lemma}

\begin{proof}
Let $\nu$ be an $\theta$-valuation and again define $s_i = \nu(y_i)$
and $t_i = \nu(z_i)$ for $1 \le i \le m$.
By Lemma~\ref{lem:pq}\eqref{lem:endpoints} we can restrict the choice of the left endpoint
$q_0$ of the ray $\bm{q}$ to the set 
$\{ \sigma_{\nu}(r_i,s_i), \sigma_{\nu}(r_i,t_i) \mid 1 \le i \le m \}$.
We guess an index $r_0 \in \{r_1, \dots, r_m\}$,
a variable $y_0 \in V_\theta$ and a length $\ell \in \N$,
and verify that the ray $\bm{q} = (q_i)_{0 \le j \le \ell}$
with period $h_C$ and $q_0 = \sigma_\nu(r_0, \nu(y_0))$
satisfies the two conditions.
The formula $\mathsf{Pack}_\theta$ has the form
\begin{equation}
	\label{eq:pack-formula}
	\bigvee_{r_0,y_0} \exists x \geq 0 : ( \phi_1 \wedge \phi_2 ).
\end{equation}
where $\phi_1$ and $\phi_2$ are constructed in the following,
stating that conditions 1.~and 2., respectively, hold for the ray $\bm{q}$.
Note that the variable $x$ in \eqref{eq:pack-formula} stands for the value $\ell$ in the ray $\bm{q} = (q_i)_{0 \le j \le \ell}$.

\subparagraph{Condition 1.}
By Lemma~\ref{lem:pq}\eqref{lem:packed} we can express that $\sigma_{\nu}(r_i,s_i,t_i)$
is packed into $\bm{q}$ by stating that
$\sigma_\nu(r_i,s_i)$ and $\sigma_\nu(r_i,t_i)$ belong to $\supp(\bm{q})$:
\begin{equation}
\label{eq:c1-a}
\begin{aligned}
	& \ \exists 0 \le y \le x  : \sigma(E_{r_0}(y_0))   \, h_C^{y} = \sigma(E_{r_i}(y_i)) \\
	 \wedge & \ \exists 0 \le y \le x : \sigma(E_{r_0}(y_0))  \, h_C^{y} = \sigma(E_{r_i}(z_i))
\end{aligned}
\end{equation}
and that $\sigma_\nu(r_i,s_i) \, \sigma(u_{r_i})^{-1}$ and $\sigma_\nu(r_i,t_i) \, \sigma(u_{r_i})$
belong to $\overline{\supp}(\bm{q})$:
\begin{equation}
\label{eq:c1-b}
\begin{aligned}
	\big(\, &\exists y < 0 :  && \!\!\! \sigma(E_{r_0}(y_0))  \, h_C^{y}  = \sigma(E_{r_i}(y_i)) \, \sigma(u_{r_i})^{-1} ~ \vee \\
  & \exists y > x : && \!\!\!  \sigma(E_{r_0}(y_0))  \, h_C^{y}  = \sigma(E_{r_i}(y_i)) \, \sigma(u_{r_i})^{-1}\big) \\
\wedge	~ \big(\, & \exists y < 0 :  && \!\!\! \sigma(E_{r_0}(y_0))  \, h_C^{y}  = \sigma(E_{r_i}(z_i)) \, \sigma(u_{r_i}) ~ \vee \\
 & \exists y > x : && \!\!\! \sigma(E_{r_0}(y_0))  \, h_C^{y}  = \sigma(E_{r_i}(z_i)) \, \sigma(u_{r_i})\big).
\end{aligned}
\end{equation}
Using the representation $h_C = \prod_{r \in C} \sigma(u_r)^{\alpha_r}$ we can write
the term $h_C^y$ as $\prod_{r \in C} \sigma(u_r)^{\alpha_r y}$ (recall that the $\sigma(u_r)$ for $r \in C$ 
pairwise commute). The formula $\phi_1$ is the conjunction of the above formulas \eqref{eq:c1-a} and  \eqref{eq:c1-b} for all $1 \leq i \leq m$.

\subparagraph{Condition 2.}
Next we will express condition 2.~under the assumption that condition 1.~already holds.
Let $1 \le i \le m$.
Since $\sigma_{\nu}(r_i,s_i,t_i)$ is packed into $\bm{q}$ and $\sigma(u_{r_i}) = h_C^{\beta_{r_i}}$,
by Lemma~\ref{lem:pq}\eqref{lem:mod} there exists a unique number $0 \le \gamma_i < \beta_{r_i}$
such that we have
\begin{equation}
	\label{eq:gamma-i}
	\supp(\sigma_{\nu}(r_i,s_i,t_i)) = \{ q_j \mid 0 \le j \le \ell, \, j \equiv \gamma_i \!\! \pmod{\beta_{r_i}} \},
\end{equation}
which is equivalent (due to condition 1.) to
\begin{equation}
	\label{eq:gamma-cond}
	q_{\gamma_i} = q_0 \, h_C^{\gamma_i} \in \supp(\sigma_{\nu}(r_i,s_i,t_i))
	= \{ \sigma_{\nu}(r_i,k) \mid s_i \le k \le t_i \}.
\end{equation}
Since we have the bound $\beta_{r_i}$ we can guess and verify these numbers $\gamma_i$.
Consider a tuple $\bm{\gamma} = (\gamma_1, \dots, \gamma_m) \in \N^m$.
For $1 \le i \le m$ we define the function $f_{\bm{\gamma},i} \colon \N \to A$ by
\begin{equation}
	f_{\bm{\gamma},i}(j) = \begin{cases} a_{r_i}, & \text{if } j \equiv \gamma_i \!\! \pmod{\beta_{r_i}}, \\
	0, & \text{otherwise.}\end{cases}
\end{equation}
By \eqref{eq:shift-sum} and \eqref{eq:gamma-i} we have for all $0 \le j \le \ell$:
\begin{equation}
\begin{aligned}
\tau_{\nu}(r_i,s_i,t_i)(q_j)
= \left. \begin{cases} a_{r_i}, & \text{if } q_j \in \supp(\sigma_{\nu}(r_i,s_i,t_i)) \\ 0, & \text{otherwise}\end{cases}
\right\} = f_{\bm{\gamma},i}(j)
\end{aligned}
\end{equation}
Hence, for all $0 \le j \le \ell$ we have
\begin{equation}
	\label{eq:tau-f}
	\sum_{i=1}^m \tau_{\nu}(r_i,s_i,t_i)(q_j) = \sum_{i=1}^m f_{\bm{\gamma},i}(j) =: f_{\bm{\gamma}}(j).
\end{equation}
Since $f_{\bm{\gamma},i}$ is $\beta_{r_i}$-periodic, Proposition~\ref{prop:nilpotent-recurrence} implies that the number
\[ b_{\bm{\gamma}} := \sup \{ j \in \N \mid f_{\bm{\gamma}}(j') = 0 \text{ for all } 0 \le j' \le j \}\]
is either infinite or bounded polynomially in $\max\{\beta_{r_i} \mid 1 \leq i \leq m\} \le 2^{n^{\O(1)}}$.
Hence, with \eqref{eq:gamma-i} and \eqref{eq:tau-f} we get
\begin{eqnarray*}
\sum_{i=1}^m \tau_{\nu}(r_i,s_i,t_i) = 0 & \iff & \forall j \in [0,\ell] : \sum_{i=1}^m \tau_{\nu}(r_i,s_i,t_i)(q_j) = 0 \\
& \iff & \forall j \in [0,\ell] : f_{\bm{\gamma}}(j) = 0 \\
& \iff & \ell \le b_{\bm{\gamma}} .
\end{eqnarray*}
The formula $\phi_2$ can now be defined by guessing the numbers $\gamma_i$ (bounded by $\beta_{r_i}-1$),
verifying them using \eqref{eq:gamma-cond} and testing that $\ell$ is at most $b_{\bm{\gamma}}$:
\begin{equation}
	\label{eq:c2}
	\phi_2 = \bigvee_{\bm{\gamma}} \Big(x \le b_{\bm{\gamma}} \wedge \bigwedge_{i=1}^m
	\exists z  \, \exists y_i \le y \le z_i : \big(\sigma(E_{r_0}(y_0)) \, h_C^z = \sigma(E_{r_i}(y))
	\wedge z = \gamma_i\big) \Big)
\end{equation}
Notice that at the atomic level the formula $\mathsf{Pack}_\theta$ 
consists of (in)equalities and exponent equations over $H$
(see \eqref{eq:pack-formula}, \eqref{eq:c1-a}, \eqref{eq:c1-b} and \eqref{eq:c2}).
The exponent equations over $H$ define semilinear sets with magnitude $2^{n^{\O(1)}}$
by Lemma~\ref{lem:exp-eq} (we need the exponents $k_i$ in Lemma~\ref{lem:exp-eq} because of the exponents $\alpha_r$
in $h_C$).
The coefficients in the (in)equalities are also bounded by $2^{n^{\O(1)}}$.
By pushing conjunctions inside 
we can transform $\mathsf{Pack}_\theta$ into a disjunction of existential formulas
of size with $\O(n+m)$ many variables and conjunctions of length $\O(m)$.
By Lemma~\ref{lem:cap} the semilinear set defined by $\mathsf{Pack}_\theta$
has magnitude $2^{(n+m)^{\O(1)}}$.
\end{proof}

\begin{proof}[Proof of Theorem~\ref{thm:abelian-wr}]
We express the statement from Proposition~\ref{prop:main} using 
Lemma~\ref{lem:stack} and Lemma~\ref{lem:pack}.
First we guess the total number $k = \O(n^6)$ of progressions.
Let $Y_k = \{y_1,z_1,\dots,y_k,z_k\}$ be a set of $2k$ distinct variables.
We then guess a set $\Theta$ of bundle descriptors such that
$\{V_\theta \mid \theta \in \Theta\}$ forms a partition of $Y_k$.
In particular, the size of $\Theta$ is bounded by $k = \O(n^6)$.
The final formula then is:
\[
	\sigma(E) = 1 \wedge \bigvee_{k,\Theta} \exists y_1 \exists z_1 \cdots \exists y_k \exists z_k
	\left( \mathsf{Decomp} \wedge \bigwedge_{\theta \in \Theta} (\mathsf{Stack}_\theta \vee \mathsf{Pack}_\theta) \right)
\]
Here the formula $\mathsf{Pack}_\theta$ should be interpreted as false
if the $r_i$-values in $\theta$ are not contained in a common $\parallel$-class.
The formula $\mathsf{Decomp}$ expresses that for all $1 \le r \le m$ 
the set $\{ [\nu(y),\nu(z)] \mid (r,y,z) \in \theta \in \Theta \}$
constitutes a partition of $[1,\nu(x_r)]$,
which is a semilinear constraint with constant magnitude.
\end{proof}

\section{Wreath products with difficult knapsack and power word problems}
  
In this section we will prove Theorems~\ref{cor:wreath-SENS-power} and \ref{cor:wreath-SENS} and present some applications.
We start with a formal definition of uniformly SENS groups \cite{BFLW20}.
	
\subsection{Strongly efficiently non-solvable groups} \label{sec-SENS}

Let us fix a f.g.~group $G = \langle \Sigma \rangle$. Following \cite{BFLW20} 
we need the additional assumption that the generating set $\Sigma$ contains the group identity $1$.
This allows to pad words over $\Sigma$ to any larger length without changing the group element 
represented by the word. One also says that $\Sigma$ is a {\em standard generating set} for $G$.
The group $G$ is called
  \emph{strongly efficiently non-solvable} \emph{(SENS)} if there is a constant $\mu\in\N$ 
  such that for every $d \in \mathbb{N}$ and $v \in \{0,1\}^{\leq d}$ there is a  word
  $w_{d,v} \in \Sigma^*$ with the following properties:
  \begin{itemize}
  \item $|w_{d,v}| = 2^{\mu d}$ for all $v \in \{0,1\}^{d}$,
  \item $w_{d,v} = [w_{d,v0}, w_{d,v1}]$ for all $v \in \{0,1\}^{< d}$ (here we take the commutator of words),
  \item $w_{d,\varepsilon} \neq 1$ in $G$.
  \end{itemize}
   The group $G$ is called \emph{uniformly strongly efficiently
    non-solvable} if, moreover,
  \begin{itemize}
  \item given $v\in \{0,1\}^d$, a binary number $i$ with $\mu d$ bits, and $a \in \Sigma$ one can decide in 
  linear time on a random access Turing-machine whether the $i$-th letter of $w_{d,v}$ is $a$.
  \end{itemize} 
Here are examples for uniformly SENS groups; see \cite{BFLW20} for details:
\begin{itemize}
\item finite non-solvable groups (more generally, every f.g.~group that has a finite non-solvable quotient),
\item f.g.~non-abelian free groups,
\item Thompson's group $F$,
\item weakly branched self-similar groups with a f.g.~branching subgroup (this includes several famous 
self-similar groups like the Grigorchuk group, the Gupta-Sidki groups and the Tower of Hanoi groups).
\end{itemize}

\subsection{Applications of Theorems~\ref{cor:wreath-SENS}}

Recall that Theorem~\ref{cor:wreath-SENS} states that $\KP(G \wr \Z)$ is $\Sigma^p_2$-hard
for every uniformly SENS group $G$. Before we prove this results we show some applications.
  
\begin{corollary}
For the following groups $G$, $\KP(G \wr \mathbb{Z})$ is $\Sigma_2^p$-complete:
\begin{itemize}
\item finite non-solvable groups,
\item non-elementary hyperbolic groups.\footnote{A hyperbolic group is non-elementary if it is not virtually cyclic. Every 
non-elementary hyperbolic group contains a non-abelian free group.}
\end{itemize}
\end{corollary}

\begin{proof}
Finite non-solvable groups and  f.g.~non-abelian free groups are uniformly SENS \cite{BFLW20}.
By Theorem~\ref{cor:wreath-SENS}, $\KP(G \wr \mathbb{Z})$ is $\Sigma_2^p$-hard.
It remains to show that $\KP(G \wr \mathbb{Z})$ belongs to $\Sigma_2^p$.
According to \cref{prop:NP-reduction}, it suffices to show that $\textsc{Periodic}(G)$ and $\ExpEq(G)$ both belong to $\Sigma_2^p$.
The problem $\textsc{Periodic}(G)$ belongs to $\coNP$ (since the word problem for $G$ can be solved in polynomial time)
and $\ExpEq(G)$ belongs to $\NP$. For a finite group this is clear. If $G$ is hyperbolic, then one can reduce
$\ExpEq(G)$ to the existential fragment of Presburger arithmetic using \cite{Loh19hyp}.
\end{proof}
Theorem~\ref{cor:wreath-SENS} can be also applied to Thompson's group $F$. This is one of the most
well studied groups in (infinite) group theory due to its unusual properties, see e.g. \cite{CaFlPa96}.
It can be defined in several ways; let us just mention the following finite presentation:
$F = \langle x_0, x_1 \mid [x_0 x_1^{-1} \!\!\:,\,  x_0^{-1} x_1 x_0],  [x_0 x_1^{-1} \!\!\:,\, x_0^{-2} x_1 x^2_0]   \rangle$. 
Thompson's group $F$ is uniformly SENS \cite{BFLW20} and contains a copy of $F \wr \mathbb{Z}$ \cite{GubaSapir99}.
Theorem~\ref{cor:wreath-SENS} yields

\begin{corollary}
The knapsack problem for Thompson's group $F$ is $\Sigma_2^p$-hard.
\end{corollary}
We conjecture that the knapsack problem for $F$ is in fact $\Sigma_2^p$-complete.

\subsection{Proof of Theorems~\ref{cor:wreath-SENS}}

We prove Theorem~\ref{cor:wreath-SENS} in two steps. The second step works for every f.g.~group $G$.
Fix this group $G$ and let $\Sigma$ be a standard
generating set for $G$.
Let $\overline{X} = (X_1, \ldots, X_n)$ be a tuple of boolean variables. We identify $\overline{X}$ with the set $\{X_1, \ldots, X_n\}$
when appropriate.
A $G$-program with variables from $\overline{X}$ is a sequence
\[
P = (X_{i_1},a_1,b_1) (X_{i_2},a_2,b_2) \cdots (X_{i_\ell},a_\ell,b_\ell) \in (\overline{X} \times \Sigma \times \Sigma)^*.
\]
The length of $P$ is $\ell$.
For a mapping $\alpha : \overline{X} \to \{0,1\}$ (called an assignment) we define
$P(\alpha) \in G$ as the group element $c_1 c_2 \cdots c_\ell$, where $c_j = a_j$ if $X_{i_j} = 1$ and 
$c_j = b_j$ if $X_{i_j} = 0$ for all $1 \leq j \leq \ell$.
We define the following computational problem $\exists\forall$-\SAT$(G)$:
\begin{description}
\item[Input] A $G$-program $P$ with variables from $\overline{X} \cup \overline{Y}$, where $\overline{X}$ and $\overline{Y}$ are disjoint.
\item[Question] Is there an assignment $\alpha : \overline{X} \to \{0,1\}$
such that for every assignment $\beta : \overline{Y} \to \{0,1\}$ 
we have $P(\alpha \cup \beta) = 1$ (we write $\exists \overline{X} \forall \overline{Y} : P=1$ for this)?
\end{description}

\begin{lemma} \label{lemma:SAT(G)}
 Let the f.g.~group $G = \langle \Sigma\rangle$ be uniformly SENS. Then, $\exists\forall$-\SAT$(G)$ is $\Sigma^p_2$-hard.
\end{lemma}

\begin{proof}
We prove the lemma by a reduction from the following $\Sigma^p_2$-complete problem: given a 
boolean formula $F = F(\overline{X}, \overline{Y})$ in disjunctive normal form, where $\overline{X}$ and $\overline{Y}$ are disjoint tuples of boolean variables, 
does the quantified boolean formula $\exists \overline{X} \forall \overline{Y} : F$ hold?
Let us fix such a formula $F(\overline{X}, \overline{Y})$. 
We can write $F$ as a fan-in two boolean circuit of depth $\mathcal{O}(\log |F|)$. 
 By \cite[Remark~34]{BFLW20} we can compute in logspace
from $F$ a $G$-program $P$ over the variables $\overline{X} \cup \overline{Y}$ 
of length polynomial in $|F|$ such that for every assignment $\gamma : \overline{X} \cup \overline{Y} \to \{0,1\}$ the following two
statements are equivalent:
\begin{itemize}
\item $F(\gamma(\overline{X}), \gamma(\overline{Y}))$ holds.\label{Cval}
\item  $P(\gamma) = 1$ in $G$. \label{Pval}
\end{itemize}
Hence, $\exists \overline{X} \forall \overline{Y} : F$ holds if and only if $\exists \overline{X} \forall \overline{Y} : P=1$
holds.
\end{proof}

\begin{lemma} \label{lemma:reduction}
For every f.g.~nontrivial group $G$, $\exists\forall$-\SAT$(G)$ is logspace many-one reducible to $\KP(G \wr \mathbb{Z})$.
\end{lemma}

  \begin{proof}
	Let us fix a $G$-program
	 \begin{equation} \label{eq-G-prog}
             P = (Z_{1},a_1,b_1) (Z_{2},a_2,b_2) \cdots (Z_{\ell},a_\ell,b_\ell) \in ((\overline{X} \cup \overline{Y}) \times \Sigma \times \Sigma)^*
         \end{equation}
         where $\overline{X}$ and $\overline{Y}$ are disjoint sets of variables.
          Let	$m = |\overline{X}|$ and $n = |\overline{Y}|$.
	 We want to construct a knapsack expression $E$ over $G \wr \mathbb{Z}$ which has a solution 
	if and only if there is an assignment $\alpha : \overline{X} \to \{0,1\}$
        such that $P(\alpha \cup \beta) = 1$ for every assignment $\beta : \overline{Y} \to \{0,1\}$.
        Let us choose a generator $t$ for $\mathbb{Z}$. Then $\Sigma \cup \{t,t^{-1}\}$ generates the wreath product $G \wr \mathbb{Z}$.	
        First, we compute in logspace the 
	$m+n$ first primes $p_1, \ldots, p_{m+n}$ and fix a bijection $p : \overline{X} \cup \overline{Y} \to \{p_1, \ldots, p_{m+n}\}$. 
	Moreover, let $M = \prod_{i=1}^{m+n} p_i$.

        Roughly speaking, the idea is as follows. Each assignment
        $\alpha\colon \overline{X}\to\{0,1\}$ will correspond to a
        valuation $\nu$ for our expression $E$. The resulting element
        $\nu(E)\in G\wr\Z$ then encodes the value $P(\alpha\cup\beta)$
        for each $\beta\colon\overline{Y}\to\{0,1\}$ in some position $s\in[0,M-1]$.  To be
        precise, to each $s\in[0,M-1]$, we associate the assignment
        $\beta_s\colon\overline{Y}\to\{0,1\}$ where $\beta_s(Y)=1$ if
        and only if $s\equiv 0\bmod{p(Y)}$. Then, $\tau(\nu(E))(s)$
        will be $P(\alpha\cup\beta_s)$.  This means, $\nu(E)=1$
        implies that $P(\alpha\cup\beta)=1$ for all assignments
        $\beta\colon\overline{Y}\to\{0,1\}$.

        Our expression implements this as follows. For each
        $i=1,\ldots,\ell$, it walks to the right to some position
        $M'\ge M$ and then walks back to the origin. On the way to the
        right, the behavior depends on whether $Z_i$ is an existential
        or a universal variable. If $Z_i$ is existential, we either
        place $a_i$ at every position (if $\alpha(Z_i)=1$) or $b_i$ at
        every position (if $\alpha(Z_i)=0$). If $Z_i$ is universal, we
        place $a_i$ in the positions divisible by $p(Z_i)$; and we
        place $b_i$ in the others. That way, in position
        $s\in[0,M-1]$, the accumulated element will be
        $P(\alpha\cup\beta_s)$.

	We define $I_\exists = \{ i \in [1,\ell] \mid Z_i \in \overline{X}\}$ and $I_\forall = \{ i \in [1,\ell] \mid Z_i \in \overline{Y} \}$.
	For an existentially quantified variable $X\in \overline{X}$ let $I_X = \{ i \in [1,\ell] \mid X = Z_i \}$ be the set of those positions in the $G$-program $P$,
         where the variable $X$ is queried.
	Moreover, let us write $q_i$ for the prime number $p(Z_i)$.
	We compute  for every $i \in I_\exists$ the words (over the wreath product $G\wr \mathbb{Z}$)
	\[
	u_{i} = (a_it)^{q_{i}} \text{ and } v_{i} = (b_it)^{q_{i}} 
	\]
        and for every $i \in I_\forall$ the word 
	\[
	w_i = a_it (b_it)^{q_{i}-1} .
	\]
	Let us now consider the knapsack expression 
	\[
	E_1 = \prod_{i=1}^\ell f_{i} t^{-1} (t^{-1})^{z_i} \text{ with } f_{i} = 
	\begin{cases} 
	   u_{i}^{x_{i}} v_{i}^{x'_{i}} & \text{ if } i \in I_\exists, \\
	   w_i^{y_i} & \text{ if } i \in I_\forall  .
	\end{cases}
      \]
      The idea is that in $E_1$, for each $i\in[1,\ell]$, we go to
      right with $f_i$ and then we go back to the origin with
      $t^{-1}(t^{-1})^{z_i}$. If $Z_i$ is existential, we use
      $f_i=u_i^{x_i}v_i^{x'_i}$ to either place $a_i$ at every
      position or $b_i$ at every position. If $Z_i$ is universal, we
      use $w_i$ to place $a_i$ at positions divisible by $q_i=p(Z_i)$
      and $b_i$ at the others.  Note that the expression itself cannot
      guarantee that, e.g., (i)~$(t^{-1})^{z_i}$ moves exactly onto
      the origin or (ii)~that we either use only $u_i$ or only $v_i$
      for each $i\in I_\exists$. Therefore, we ensure these properties
      temporarily by imposing additional linear equations (Claim~1). In a
      second step, we shall extend $E_1$ to get an expression in which
      a solution will automatically satisfy these linear equations
      (Claim~2).
      
      \medskip
	\noindent
	{\it Claim 1:} $\exists \overline{X} \forall \overline{Y} : P=1$ holds if and only if there exists a $(G \wr \mathbb{Z})$-solution $\nu$ for $E_1$ with the following
	properties:
	\begin{enumerate}[(a)]
        \item $q_i \cdot \nu(y_i) = \nu(z_i)+1$ for all $i \in I_\forall$,
        \item $q_i \cdot (\nu(x_{i}) + \nu(x'_{i})) = \nu(z_i)+1$ for all $i \in I_\exists$,
        \item $\nu(z_i) = \nu(z_j)$ for all $i,j \in [1,\ell]$ with $i \neq j$,
        \item $\nu(x_{i}) = 0$ or $\nu(x'_{i}) = 0$ for all $i \in I_\exists$,
        \item for all $X \in \overline{X}$ and all $i,j \in I_X$ 
         we have: $\nu(x_{i}) = 0$ if and only if $\nu(x_{j}) = 0$.
        \end{enumerate}
        
        \noindent
        {\it Proof of Claim 1:} Assume first that $\exists \overline{X} \forall \overline{Y} : P=1$ holds.
        Let  $\alpha :  \overline{X} \to \{0,1\}$ be an assignment such that for every assignment $\beta :  \overline{Y} \to \{0,1\}$, we have
        $P(\alpha \cup \beta)=1$ in $G$. 
        
        We have to find a $(G \wr \mathbb{Z})$-solution for $E_1$ such that the above properties (a)--(d) hold. 
        For this, we set:
        \begin{itemize}
        \item $\nu(z_i) = M-1$ for all $i \in [1,\ell]$,
        \item  $\nu(y_i) = M/q_i$ for all $i \in I_\forall$,
        \item  $\nu(x_i) = M/q_i$ and $\nu(x'_i) = 0$
        for all $i \in I_X$, $X \in \overline{X}$ such that $\alpha(X)=1$,
        \item $\nu(x'_i)= M/q_i$ and $\nu(x_i) = 0$
        for all $i \in I_X$, $X \in \overline{X}$ such that $\alpha(X)=0$.
        \end{itemize}
        Then, clearly, (a)--(e) hold. It remains to verify that $\nu$ is a $(G \wr \mathbb{Z})$-solution for $E_1$. 
        Let $h = \tau(\nu(E_1)) \in G^{(\Z)}$ and $k = \sigma(\nu(E_1)) \in \Z$.
	We have $k=0$ and $h(s) = 1$ for all $s \in \mathbb{Z} \setminus [0,M-1]$. 
	Moreover, for every $s \in [0,M-1]$ we have
	$h(s) = c_1 c_2 \ldots c_\ell$ where 
	\[ c_i = \begin{cases}
	    a_i & \text{ if } (i \in I_\forall \text{ and } s \equiv 0 \bmod q_i) \text{ or } (i \in I_X, X \in\overline{X} \text{ and } \alpha(X)=1) \\
	    b_i & \text{ if } (i \in I_\forall \text{ and } s \not\equiv 0 \bmod q_i) \text{ or } (i \in I_X, X \in\overline{X} \text{ and } \alpha(X)=0) .
	 \end{cases}  
	\]	
	Here, the $a_i$ and $b_i$ are from \eqref{eq-G-prog}.
	Hence, there is an assignment $\beta_s :  \overline{Y} \to \{0,1\}$ such that
	$h(s) = P(\alpha \cup \beta_s)$. Thus, $h(s) = 1$ for all $s \in [0,M-1]$, which implies 
	that $\nu(E_1)=1$ in $G \wr \mathbb{Z}$.
	
	For the other direction, assume that $\nu$ is a $(G \wr \mathbb{Z})$-solution for $E_1$ such that the properties
	(a)--(e) hold. Let $M' = \nu(z_1)+1>0$. We then have $M' = \nu(z_i)+1$ for all $i \in [1,\ell]$
	by property (c). By properties (a) and (b), $M'$ is divisible by the first $m+n$ primes. This
	implies that $M'$ is a multiple of $M$ and thus $M' \geq M$. 

	Let us define an assignment $\alpha :  \overline{X} \to \{0,1\}$ as follows, where $i \in I_\exists$: 
	\[
	\alpha(Z_i) = \begin{cases}
	  0 \text{ if } \nu(x_i) = 0 \\
	  1 \text{ if } \nu(x'_i) = 0
	\end{cases}
	\]
	By properties (d) and (e) this defines indeed an assignment $\alpha :  \overline{X} \to \{0,1\}$.
	Moreover, for every position $s \in [0,M'-1]$ we define the assignment $\beta_s : \overline{Y} \to \{0,1\}$
	by $\beta_s(Y) = 1$ if $s \equiv 0 \bmod p(Y)$ and  $\beta_s(Y) = 0$ otherwise.
	By the Chinese remainder theorem, for every $\beta : \overline{Y} \to \{0,1\}$ there exists
	$s \in [0,M'-1]$ with $\beta = \beta_s$. Moreover, the construction of $E_1$ implies that 
	$\nu(E_1)$ writes $P(\alpha \cup \beta_s)$ into position $s$. Since $\nu(E_1) = 1$ in $G \wr \mathbb{Z}$
	we have $P(\alpha \cup \beta_s) = 1$ for all $s \in  [0,M'-1]$, i.e., $P(\alpha \cup \beta) = 1$
	for all assignments $\beta : \overline{Y} \to \{0,1\}$. 
	We have shown Claim 1.
	
	\medskip
	\noindent
	In the rest of the proof we construct a knapsack expression $E_2$ such that 
	each of the variables from $E_1$ also occurs in $E_2$. Moreover, the following properties will hold:
	\begin{itemize}
        \item Every $(G \wr \mathbb{Z})$-solution of $E_1$ that satisfies the properties (a)--(e) extends to a $(G \wr \mathbb{Z})$-solution 
        of $E_2$.
        \item Every $(G \wr \mathbb{Z})$-solution of $E_2$ restricts to a  $(G \wr \mathbb{Z})$-solution of $E_1$ that satisfies the properties (a)--(e). 
        \end{itemize}
        This implies that $E_2$ has a $(G \wr \mathbb{Z})$-solution if and only if $E_1$ has a   
         $(G \wr \mathbb{Z})$-solution that satisfies the properties (a)--(e) if and only if $\exists \overline{X} \forall \overline{Y} : P=1$ holds.
         
         Let $g \in G$ be any nontrivial element. 
         To construct $E_2$ it is convenient to work in a wreath product $(\langle g \rangle^d \times G) \wr \mathbb{Z}$ for
         some $d$, whose unary encoding can be computed (in logspace) from the input formula $\exists \overline{X} \forall \overline{Y} : F$.
         By Lemma~\ref{lemma-wreath-embedding} we can compute in logspace an embedding of $(\langle g \rangle^d \times G) \wr \mathbb{Z}$ into $G \wr \mathbb{Z}$.
         Let $\zeta_i$ be the canonical embedding of $\langle g \rangle$ into $\langle g \rangle^d$ that maps $g$ to $(1,\ldots,1,g,1,\ldots,1)$,
         where in the latter, $g$ appears in the $i$-th coordinate. We assume that the coordinates are numbered from $0$ to $d-1$. 
         In the following, we write $g_i$ for $\zeta_i(g)$. We set $d = 2 \ell+1$.
         
         We then define the following knapsack expression $E_2 = E_{2,1} E_{2,2}$ where $z,z'$ and $\tilde{X}, \tilde{X}'$ for all $X \in \overline{X}$
         appear as fresh variables:
         \begin{eqnarray*}
         E_{2,1} &=& g_0 g_1 \cdots g_\ell \left( \prod_{X \in \overline{X}} \bigg( \prod_{i \in I_X} g_{\ell+i} \bigg)^{\! \tilde{X}'}\right) t t^z g_1 \cdots g_\ell 
         \left( \prod_{X \in \overline{X}} \bigg( \prod_{i \in I_X} g_{\ell+i} \bigg)^{\! \tilde{X}}\right) t^{-1} (t^{-1})^{z'} g_0^{-1}\\
         E_{2,2} &=&  \prod_{i =1}^\ell f_{i} g_i^{-1} t^{-1} (t^{-1})^{z_i} g_i^{-1}  
         \text{ with } f_{i} = 
	\begin{cases} 
	   u_{i}^{x_{i}} g_{\ell+i}^{-1} v_{i}^{x'_{i}} & \text{ if } i \in I_\exists, \\
	   w_i^{y_i} & \text{ if } i \in I_\forall .	   
	\end{cases}
         \end{eqnarray*}
         The idea of the construction is that the $g_i$ implement pebbles that can be put on different positions in $\mathbb{Z}$. At the end all pebbles have to be recollected.
         Note that we only use the pebbles $g_0, g_1, \ldots, g_\ell$ and $g_{\ell+i}$ for $i \in I_\exists$; hence we could reduce the dimension $2\ell+1$ to
         $\ell+1+|I_\exists|$ but this would make the indexing slightly more inconvenient.
         
         \medskip
         \noindent
         {\it Claim 2:}  Every $(G \wr \mathbb{Z})$-solution of $E_1$ that satisfies the properties (a)--(e) extends to a $(G \wr \mathbb{Z})$-solution 
        of $E_2$.
        
         \medskip
        \noindent
        {\it Proof of Claim 2:} 
        Let $\nu$ be a $(G \wr \mathbb{Z})$-solution of $E_1$  that satisfies the properties (a)--(e).
        Let $M' = \nu(z_1)+1>0$. Hence, $M' = \nu(z_i)+1$ for all $i \in [1,\ell]$.
        We then extend $\nu$ to the fresh variables in $E_2$ by: 
        \begin{itemize}
         \item $\nu(z) = \nu(z')=M'-1$,
         \item for all $X \in \overline{X}$ such that $x_i=0$ for some (and hence all) $i \in I_X$, we set $\nu(\tilde{X}')=1$ and $\nu(\tilde{X})=0$,
         \item  for all $X \in \overline{X}$ such that $x'_i=0$ for some (and hence all) $i \in I_X$, we set $\nu(\tilde{X}')=0$ and $\nu(\tilde{X})=1$.   
        \end{itemize}
	It is easy to check that this yields indeed a $(G \wr \mathbb{Z})$-solution 
        of $E_2$.

          \medskip
         \noindent
         {\it Claim 3:}  Every $(G \wr \mathbb{Z})$-solution of $E_2$ restricts to a  $(G \wr \mathbb{Z})$-solution of $E_1$ that satisfies the properties (a)--(e). 
                 
         \medskip
        \noindent
        {\it Proof of Claim 3:} Fix a $(G \wr \mathbb{Z})$-solution $\nu$ of $E_2$.  First of all, we must have $\nu(z)=\nu(z')$; otherwise the pebble $g_0$
        will not be recollected. Let $M' = \nu(z)+1>0$. The word $\nu(E_{2,1})$ leaves pebbles $g_1, \ldots, g_\ell$ at positions $0$ and $M'$ (it also leaves powers
        of the pebbles $g_{\ell+i}$ --- we will deal with those later) and puts the cursor back to position $0$. With the word $\nu(E_{2,1})$ the pebbles 
        at positions $0$ and $M'$ have to be recollected. This happens only if 
        $\nu(z_i) = M'-1$ for all $i \in [1,\ell]$, $q_i \cdot \nu(y_i) = M'$ for all $i \in I_\forall$,
        and $q_i \cdot (\nu(x_{i}) + \nu(x'_{i})) = M'$ for all $i \in I_\exists$. Hence, conditions (a)--(c) hold.
        
        Conditions (d) and (e) are enforced with the pebbles $g_{\ell+i}$ for $i \in I_\exists$. Consider an existentially quantified variable $X \in \overline{X}$.
        The word $\nu(E_{2,1})$ leaves for every $i \in I_X$ the 
        ``pebble powers'' $g_{\ell+i}^{\nu(\tilde{X}')}$ and $g_{\ell+i}^{\nu(\tilde{X})}$ at positions 
        $0$ and $M'>0$, respectively. With the word $\nu(E_{2,2})$ exactly
        one pebble $g_{\ell+i}$ is recollected. Therefore, exactly one of the following two cases has to hold:
        \begin{itemize}
         \item $g^{\nu(\tilde{X}')} = 1$ and  $g^{\nu(\tilde{X})} = g$ in $G$,
         \item $g^{\nu(\tilde{X}')} = g$ and  $g^{\nu(\tilde{X})} = 1$ in $G$.
        \end{itemize}
        Assume first that $g^{\nu(\tilde{X}')} = 1$ and  $g^{\nu(\tilde{X})} = g$ in $G$. Then $\nu(E_{2,1})$ places the pebble $g_{\ell+i}$ at position $M'$ (and it places this 
        pebble at no other position) for all $i \in I_X$.
        In order to recollect this pebble with $\nu(E_{2,2})$ we must have $\nu(x_i) = M'/q_i = M'/p(X)$ and $\nu(x'_i) = 0$ for all $i \in I_X$.
         If $g^{\nu(\tilde{X}')} = g$ and  $g^{\nu(\tilde{X})} = 1$ in $G$ then we must have $\nu(x'_i) = M'/q_i = M'/p(X)$ and $\nu(x_i) = 0$ for all $i \in I_X$.
         This shows that (d) and (e) holds and concludes the proof of Claim 3 and hence the proof of the lemma.
	\end{proof}
Theorem~\ref{cor:wreath-SENS} is now a direct corollary of Lemmas~\ref{lemma:SAT(G)} and \ref{lemma:reduction}.

\subsection{Wreath product with difficult power word problems}

In \cite{LoWe19} it was shown that $\PowWP(G\wr \Z)$ is $\coNP$-complete in case
$G$ is a finite non-solvable group or a f.g.~free group. The proof in \cite{LoWe19} immediately generalizes to the case
were $G$ is uniformly SENS. This yields Theorem~\ref{cor:wreath-SENS-power}.
Alternatively, one can prove Theorem~\ref{cor:wreath-SENS-power} by showing the following two facts:
\begin{itemize}
\item $\forall$-\SAT$(G)$ (the question whether for a given $G$-program $P$,
$P(\alpha)=1$ for all assignments)  is $\coNP$-hard if $G$ is uniformly SENS.
\item $\forall$-\SAT$(G)$ is logspace many-one reducible to $\PowWP(G \wr \mathbb{Z})$.
\end{itemize}
The proofs for these facts are in fact simplifications of the proofs for Lemmas~\ref{lemma:SAT(G)} and \ref{lemma:reduction}.

We can also show that for a large class of groups the power word problem is contained in $\coNP$.
Fix a f.g.~group $G = \langle\Sigma\rangle$. With 
$\WP(G,\Sigma)$ we denote the set of all words $w \in \Sigma^*$ such that $w=1$ in $G$ (the word problem
for $G$ with respect to $\Sigma$).
We say that $G$ is {\em co-context-free} if $\Sigma^* \setminus \WP(G,\Sigma)$
is context-free (the choice of $\Sigma$ is not relevant for this), see  \cite[Section~14.2]{HRR2017} for more details.

\begin{theorem} \label{thm-co-context-free-coNP}
The power word problem for a co-context-free group $G$ belongs to $\coNP$. 
\end{theorem}

\begin{proof}
Let $G = \langle\Sigma\rangle$ and let $(u_1, k_1, u_2, k_2, \ldots, u_d, k_d)$
be the input power word, where $u_i \in \Sigma^*$. We can assume that all $k_i$ are positive.
We have to check whether $u_1^{k_1} u_2^{k_2}\cdots u_d^{k_d}$ is trivial in $G$.
Let $L$ be the complement of $\WP(G,\Sigma)$, which is context-free. 
Take the alphabet $\{a_1, \ldots, a_d\}$ and define the morphism $h : \{a_1, \ldots, a_d\}^* \to \Sigma^*$
by $h(a_i) = u_i$.
Consider the language $K = h^{-1}(L) \cap a_1^* a_2^* \cdots a_d^*$. Since the context-free languages are closed under inverse 
morphisms and intersections with regular languages, $K$ is context-free too. Moreover, from the tuple $(u_1, u_2, \ldots, u_d)$
we can compute in polynomial time a context-free grammar for $K$: Start with a push-down automaton $M$ for $L$ (since $L$ is a fixed
language, this is an object of constant size). From $M$ one can compute in polynomial time a push-down automaton $M'$ for $h^{-1}(L)$:
when reading the symbol $a_i$, $M'$ has to simulate (using $\varepsilon$-transitions) $M$ on $h(a_i)$. Next, we construct in polynomial
time a push-down automaton $M''$ for $h^{-1}(L) \cap a_1^* a_2^* \cdots a_d^*$ using a product construction. Finally, we transform $M''$
back into a context-free grammar. This is again possible in polynomial time using the standard triple construction. It remains to check whether
$a_1^{k_1} a_2^{k_2} \cdots a_d^{k_d} \notin L(G)$. This is equivalent to $(k_1, k_2, \ldots, k_d) \notin  \Psi(L(G))$, where $\Psi(L(G))$ denotes
the Parikh image of $L(G)$. Checking $(k_1, k_2, \ldots, k_d) \in  \Psi(L(G))$ is an instance of the uniform membership problem for commutative
context-free languages, which can be solved in $\NP$ according to \cite{Hu83}. This implies that the power word problem for $G$ belongs to $\coNP$.
\end{proof}
Let us remark that the above context-free language $K$ was also used in \cite{KoenigLohreyZetzsche2015a} in order to show that the knapsack problem for 
a co-context-free group is decidable.

\begin{theorem}
For Thompson's group $F$,  the power word problem is $\coNP$-complete.
\end{theorem}

\begin{proof} 
The upper bound follows from Theorem~\ref{thm-co-context-free-coNP} and the fact that $F$ is co-context-free \cite{LehSch07}.
The lower bound follows from Theorem~\ref{cor:wreath-SENS-power} and 
the facts that $F$ is uniformly SENS and that $F \wr \mathbb{Z} \leq F$. 
\end{proof}

 \bibliography{bib}

\end{document}